\documentclass[leqno,onefignum,onetabnum]{siamltex} 
\usepackage{amssymb,amsmath,amsbsy,eucal,amsfonts,mathrsfs,latexsym}
\usepackage{graphicx,verbatim,psfrag,epsfig,enumerate}
\usepackage{lscape}
\usepackage{verbatim}
\usepackage{eepic}
\usepackage{epic}
\usepackage{epsfig}
\usepackage[]{graphics}

\newtheorem{remark}[theorem]{Remark}
\newtheorem{assumption}[theorem]{Assumption}

\newcommand{\Oh}{\mathcal{T}_h}
\newcommand{\Eh}{\mathcal{E}_h}
\newcommand{\EH}{\mathcal{E}_H}
\newcommand{\EhH}{\mathcal{E}_{h,H}}
\newcommand{\Vh}{\boldsymbol{V}_h}

\newcommand{\uhat}{\widehat{u}_{h,H}}
\newcommand{\qhat}{\widehat{\boldsymbol{q}}_{h,H}}
\newcommand{\n}{\boldsymbol{n}}
\newcommand{\Piv}{\boldsymbol{\Pi}_V}
\newcommand{\Piw}{\Pi_W}
\newcommand{\ehatq}{{\boldsymbol{e}}_{\widehat{q}}}
\newcommand{\ehatu}{e_{\widehat{u}}}
\newcommand{\TT} {  \mathcal{T}_h(T)} 

\newcommand{\bint}[2]{( #1\,,\,#2 )_{\Oh}}
\newcommand{\bintOhi}[2]{( #1\,,\,#2 )_{\TT}}
\newcommand{\bintK}[2]{( #1\,,\,#2 )_{K}}
\newcommand{\bintEh}[2]{\langle #1\,,\,#2 \rangle_{\partial{\Oh}}}

\newcommand{\bintEH}[2]{\langle #1\,,\,#2 \rangle_{\partial{\mathcal{T}_H}}}

\newcommand{\bintEK}[2]{\langle #1\,,\,#2 \rangle_{\partial K}}
\newcommand{\bintET}[2]{\langle #1\,,\,#2 \rangle_{\partial \TT}}

\def\al{{\alpha}}
\def\alphaa{{\tau}}
\def\Om{{\Omega}}

\usepackage{color}
\definecolor{black}{rgb}{0,0,0}

\definecolor{red}{rgb}{1,0,0}

\definecolor{blue}{rgb}{0,0,1}

\title{
A multiscale HDG method for second order elliptic equations. \\[1ex]
Part I. Polynomial and homogenization-based multiscale spaces\\[2ex]
}
\author{   
 Yalchin Efendiev
         \thanks{Department of Mathematics, Texas A$\&$M University, 
                 College Station, TX 77843, email: {\tt efendiev@math.tamu.edu}}
\and
 Raytcho Lazarov
         \thanks{Department of Mathematics, Texas A$\&$M University, 
                 College Station, TX 77843, email: {\tt lazarov@math.tamu.edu}}
 \and
         Ke Shi
         \thanks{Department of Mathematics, Texas A$\&$M University, 
                 College Station, TX 77843, email: {\tt shike@math.tamu.edu}}.
}

\date{started October 1, 2012, today is \today}

\begin{document}

 \maketitle

\begin{abstract}
We introduce a finite element method for numerical upscaling of second order
elliptic equations with highly heterogeneous coefficients. The method is based on
a mixed formulation of the problem and 
the concepts of the domain decomposition and 
the hybrid discontinuous Galerkin methods. 
The method utilizes three different scales: (1) the scale of the 
partition of the domain of the problem, (2) the sale of partition of
the boundaries of the subdomains (related to the corresponding space of Lagrange 
multipliers), and (3) the fine grid scale that is assumed to resolve the 
scale of the heterogeneous variation of the coefficients.
Our proposed method gives a flexible framework that (1) couples
independently generated multiscale basis functions in each coarse patch
(2) provides a stable global coupling independent of local
discretization, physical scales and contrast
(3) allows avoiding any constraints (c.f., \cite{Arbogast_PWY_07}) 
on coarse spaces.
In this paper, we develop and study a multiscale HDG method that uses
polynomial and homogenization-based multiscale spaces. 
These coarse spaces are designed for problems with scale separation.
In our consequent paper,
we plan to extend our flexible HDG framework to more challenging
multiscale problems with non-separable scales and high contrast
and consider enriched coarse spaces that use
appropriate local spectral
problems.

\end{abstract}


\section{Introduction}

In this paper we consider the following second order elliptic differential equation
for the unknown function $u(x)$
\begin{equation}\label{eq:general}
- \nabla \cdot( \kappa(x) \nabla u) = f(x), \quad x \in \Omega
\end{equation}
with homogeneous Dirichlet boundary conditions. Here $\kappa(x) \ge \kappa_0 >0$ 
and $c(x) \ge 0$ are highly heterogeneous coefficients and $\Omega$ is a bounded polyhedral
domain in $\mathbb{R}^n$, $n=2,3$. 
The presented in this paper methods
are targeting applications of equation \eqref{eq:general} to flows in porous media. Other possible applications
are diffusion and transport of passive chemicals or heat transfer in heterogeneous media.

Flows in porous media appear in many industrial, scientific, engineering, and 
environmental applications. One common characteristic of these diverse areas is 
that porous media are intrinsically multiscale and 
typically display heterogeneities over a wide range of length-scales.
Depending on the goals, solving the governing equations of flows in porous media might be sought at: 
(a) A coarse scale (e.g., if only the global pressure drop for a given flow rate is needed, and no other fine 
scale details of the solution are important), 
(b) A coarse scale enriched with some desirable fine scale details, and 
(c) The fine scale (if computationally affordable and practically desirable).

In naturally occurring materials, e.g.\ soil or rock, the permeability is small in granite formations 
(say $10^{-15}$ cm$^2$), medium in oil reservoirs, (say $10^{-7}$ cm$^2$ to $10^{-9}$ cm$^2$),  and large in 
highly fractured or in vuggy media (say $10^{-3}$ cm$^2$). 
Numerical solution of such problems is a challenging task that has attracted a 
substantial attention in the scientific and engineering community.

In the last decade a number of 
numerical upscaling schemes that fall into the class of model reduction methods 
have been developed and used in various applications in geophysics and engineering 
related to problems in highly heterogeneous media. These  include 
Galerkin multiscale finite element 
(e.g., \cite{Arb04,Chu_Hou_MathComp_10, EE03, EfendievHou_MSFEM_book,
EHW00}),
mixed multiscale finite element  
(e.g., \cite{aarnes04, ae07, Arbogast_Boyd_06,Arbogast_HomBasedMS_11}), 
the multiscale finite volume  (see, e.g., 
\cite{Durlofsky_2006,jennylt03}   
mortar multiscale  (see e.g., \cite{Arbogast_PWY_07, Arbogast_Xiao_2013},  
and variational multiscale 
(see e.g., \cite{hughes98})  methods.
%
In the paper we present a general framework to design of numerical 
upscaling method based based on subgrid approximation using the hybrid discontinuous Galerkin finite element  method
(HDG) for second order elliptic equations. 
One of the earliest subgrid (variational multiscale) methods for Darcy's problem
in a mixed form  have been developed by Arbogast  in  \cite{Arb04}, see also \cite{NPP08}.

In order to fix the main ideas and to derive the numerical upscaling method we shall consider 
model equation \eqref{eq:general} with homogeneous Dirichlet boundary condition in a mixed form:
\begin{subequations}
\begin{alignat}{2} 
\label{original equation-1}
\al \boldsymbol{q} + \nabla u &= 0 \qquad && \text{in $\Omega$,}\\
\label{original equation-2}
\nabla \cdot \boldsymbol{q}  &= f && \text{in $\Omega$}\\
\label{boundary condition}
u &= 0 && \text{on $\partial \Omega$.}
\end{alignat}
\end{subequations}
Here $\al(x)=\kappa(x)^{-1}$ $\Omega \subset \mathbb{R}^n$ ($n=2,3$) is a bounded polyhedral domain, 
$f \in L^2(\Omega)$.

In the paper  we present a multiscale finite element approximation of the mixed system 
\eqref{original equation-1} -- \eqref{boundary condition} based on the hybridized discontinuous
Galerkin method. Multiscale methods have gained substantial popularity in the last decade. We can consider them as 
a procedure of numerical upscaling that extends the capabilities of the mathematical theory of homogenization
to more general cases including materials with non-periodic properties, non-separable scales, and/or random
coefficients.

The first efficient mixed multiscale finite element methods were devised by Arbogast in \cite{Arbogast_CWY_00}
as multiblock grid approximations using the framework of mortaring technique.  Mortaring techniques
(e.g. see the pioneering work \cite{BernardiMadayPatera}) were introduced to accommodate methods that can be
defined in separate subdomains that could have been independently meshed. This
technique introduces an auxiliary space for a Lagrange multiplier associated with a 
continuity constraint on the approximate solution.  The classical mortaring,
 devised for the
needs of domain  domain decomposition methods,  has been adapted  recently  as  multiscale 
finite element approximations, e.g. \cite{ Arbogast_HomBasedMS_11, Arbogast_PWY_07, Arbogast_Xiao_2013}. 
In a two-scale (two-grid, fine and coarse) method  the aim is 
to resolve the local heterogeneities on the fine grid introduced on each coarse block and 
then "glue" these approximations together via mortar spaces, that  play the role of Lagrange
multipliers, defined of the boundaries
of the coarse partition. In order to design a stable method the mortar spaces have to
satisfy proper inf-sup condition. This approach shown to be well suited for 
problems with heterogeneous media and a number of efficient  methods and 
implementations have been proposed, studied, and used for solving a variety of 
applied problems, see, e.g. \cite{Arb04, Arbogast_HomBasedMS_11, ArB_vuggy,Arbogast_PWY_07}.

The multiscale finite element method in this paper is based on discretization of the domain $\Omega$
by using three different scales.
First the domain is split into a number of non-overlapping subdomains with characteristic size 
$L$.  This partition, denoted by $\mathcal{P}$, represents a coarse scale at which the global features of the  
 solution are captured, but the local features are not resolved.
Each subdomain is partitioned into finite elements with size $h$. This partition, 
denoted by $\mathcal{T}_h$, represents the scale 
at which the heterogeneities of the media are well represented  
and the local features of the solution can be resolved.    Finally, on 
each interface of two adjacent subdomains from the partition $\mathcal{P} $  we introduce an additional 
partition $\mathcal{E}_H$ with characteristic size $H$.
 This partition is used to introduce the space 
of the Lagrange multipliers, which will provide the means of gluing together 
of the fine-grid approximations
that are introduced on each subdomain using the fine-grid partition  $\mathcal{T}_h$.
Three scales partition of the domains have been used in the mortar multiscale finite element methods 
by Arbogast and Xiao in \cite{Arbogast_Xiao_2013}, where the scale $L$ represents the size a
cell at which a homogenized solution exits.

 The hybridization of the finite element methods as outlined in
\cite{CockburnGL2009} provides ample possibilities for "gluing" together various finite
element approximations. The mechanism of  this "glue" is based on
the notion of {\it numerical trace} and {\it numerical flux}. 
Numerical trace is 
single valued function on the finite element interfaces and belongs to certain  Lagrange
multiplier space. This is also the space at which the global problem is formulated. The stability is ensured
by a proper choice of the {\it numerical flux}, that  involves a parameter $\tau$, which proves 
stabilization of the scheme and some other desired properties (e.g. superconvergence)
(for details, see, e.g. \cite{CockburnDongGuzman2008,CockburnGL2009,CockburnQiuShi2012}).
For multiscale methods, local basis functions are constructed
independently in each coarse region.
 For this reason, approaches are needed that can flexibly 
couple these local multiscale local solutions without any constraints.
In previous works, mortar multiscale methods are proposed to couple
local basis functions; however, they require 
additional constraints on mortar spaces.
The proposed approaches can avoid any constraints on coarse spaces and provide
a flexible "gluing"  procedure for coupling multiscale basis functions.
In this paper, we focus on polynomial and homogenization
based multiscale spaces and study their stability and convergence
properties.

The paper is organized as follows. In
Sect. 2 we introduce the necessary notations and describe the multiscale FEM
based on the framework of hybridizable discontinuous Galerkin method. 
In Subsection \ref{upscaledFEM} we recast the two scale method into a hybridized form
which essentially reduces to a symmetric and positive definite system for the 
Lagrange multipliers associated with the trace of the solution of the interfaces.
In Section \ref{solvability} we show that under reasonable assumptions on the
finite dimensional spaces and the mesh, the two-scale method has unique solution.

In Section  \ref{errors} we present the error analysis for the multiscale method.
In Subsection \ref{prelims} we introduce a number of projection operators 
related to the finite dimensional spaces (that are used later) and also a special projection operator 
related to the hybridizable  discontinuous Galerkin FEM.  Further, we state the approximation properties 
of these projections in terms of the  scales of the various partitions of the domain. 
In Subsections \ref{estimate for q} and \ref{estimate for u}  we derive 
error estimates for the flux $\boldsymbol{q}$ and the pressure $u$.
Finally in  Section \ref{sec:MS}, 
we study a new class of non-polynomial space for the numerical trace for special case of heterogeneous media with
periodic  arrangement of the coefficients.
This space was proposed by Arbogast and Xiao \cite{Arbogast_Xiao_2013}
as a space for the mortar method for constructing multiscale finite element approximations
and uses  information of the local  solution on the periodic cell. 
We show that the proposed  multi-scale method is well posed and 
has proper approximation properties.

\section{Multiscale Finite Element Method}
Now we present the multiscale finite element approximation of the 
system \eqref{original equation-1}  -- \eqref{boundary condition}. For this we shall
need partition of the domain into finite elements, the corresponding finite elements spaces,
and some notation from Sobolev spaces.

\subsection{Sobolev spaces and their norms}

Throughout the paper we shall use the standard notations for Sobolev spaces and their norm on 
the domain $\Omega$, subdomains $D \subset \Omega$ or their boundaries. For example,
$\|v \|_{s,D}$, $|v |_{s,D}$, $\|v \|_{s,\partial D}$, $|v |_{s, \partial D}$, $s>0$, denote the Sobolev 
norms and semi-norms on $D$ and its boundary $\partial D$. For $s$ integer the  Sobolev spaces are 
Hilbert spaces and the norms are defined by the $L^2$-norms of their weak derivatives up to order $s$.
For $s$ non integer the spaces are defined by interpolation \cite{Grisvard85}.
For $s=0$ instead of $\|v \|_{0,D}$ we shall use $\|v \|_{D}$.  

Further, we shall use various
inequalities between norms and semi-norms  related to  embedding of Sobolev spaces. If $D \subset \Omega$
and $diam(D)=d$ then we have the following inequalities:
\begin{equation}\label{eq:trace}
\| v \|^2_{\partial D} \le C \left ( d \| \nabla  v \|^2_{D} + d^{-1} \| v \|^2_{D} \right  ) .
\end{equation}
We remark that since in this paper we are using three different scales ($L,H,h$) of partition of the domain
we shall use these inequalities for domains of sizes $L,H$  or $h$.

\subsection{Partition of the domain}

The finite element spaces that are used in the proposed method are defined below. 
They involve three different meshes.
Let $\mathcal{P}$ be a disjoint polygonal partition of the domain $\Omega$ which
allows nonconforming decomposition, see e.g. Figure \ref{fig:partition}, and
let the maximal diameter of all $T \in  \mathcal{P}$ be $L$. 
Let $\TT$ be quasi-uniform conforming triangulations of $T$ with maximum element 
diameter $h_T$, denote by $\mathcal{T}_h = \cup_{T \in \mathcal{P}} \TT$, and let $  h = \max_{T \in \mathcal{P}} h_T$. 
Let $\mathcal{E}_T$ denote the set of all edges/faces of the triangulation $\TT$ 
and $\mathcal{E}_h:= \cup_{T\in \mathcal{P}} \mathcal{E}_T$.
We also set $\partial \mathcal{T}_h = \cup_{K \in \mathcal{T}_h} \partial K$.  Consistently in this paper we shall  
denote by $T$ the subdomains of the partition $  \mathcal{P}$ while by $K$ we shall denoted the finite elements of
the fine partition $ \mathcal{T}_h$.

We call $F$ a interface of the partition $\mathcal{P}$ if $F$ is 
either shared by two neighboring subdomains, $F = \bar{T_1} \cap \bar{T_2}$ or $F = \bar{T} \cap \partial{\Omega}$.
For each interface $F$, let $\mathcal{T}_{F}$ be a quasi-uniform partition of $F$ with maximum 
element diameter $H$. Set $\mathcal{E}_H = \cup_{F \in \mathcal{E}} \mathcal{T}_{F}$ 
and  $\mathcal{E}^0_h:=\{ F \in \mathcal{E}_h: ~~ F \cap \partial T = \emptyset  \,\, \mbox{for any} \, \, T \in {\mathcal P}\}.$

\begin{figure}[ht] 
\centerline{
 \includegraphics[ width=.45\textwidth, clip]{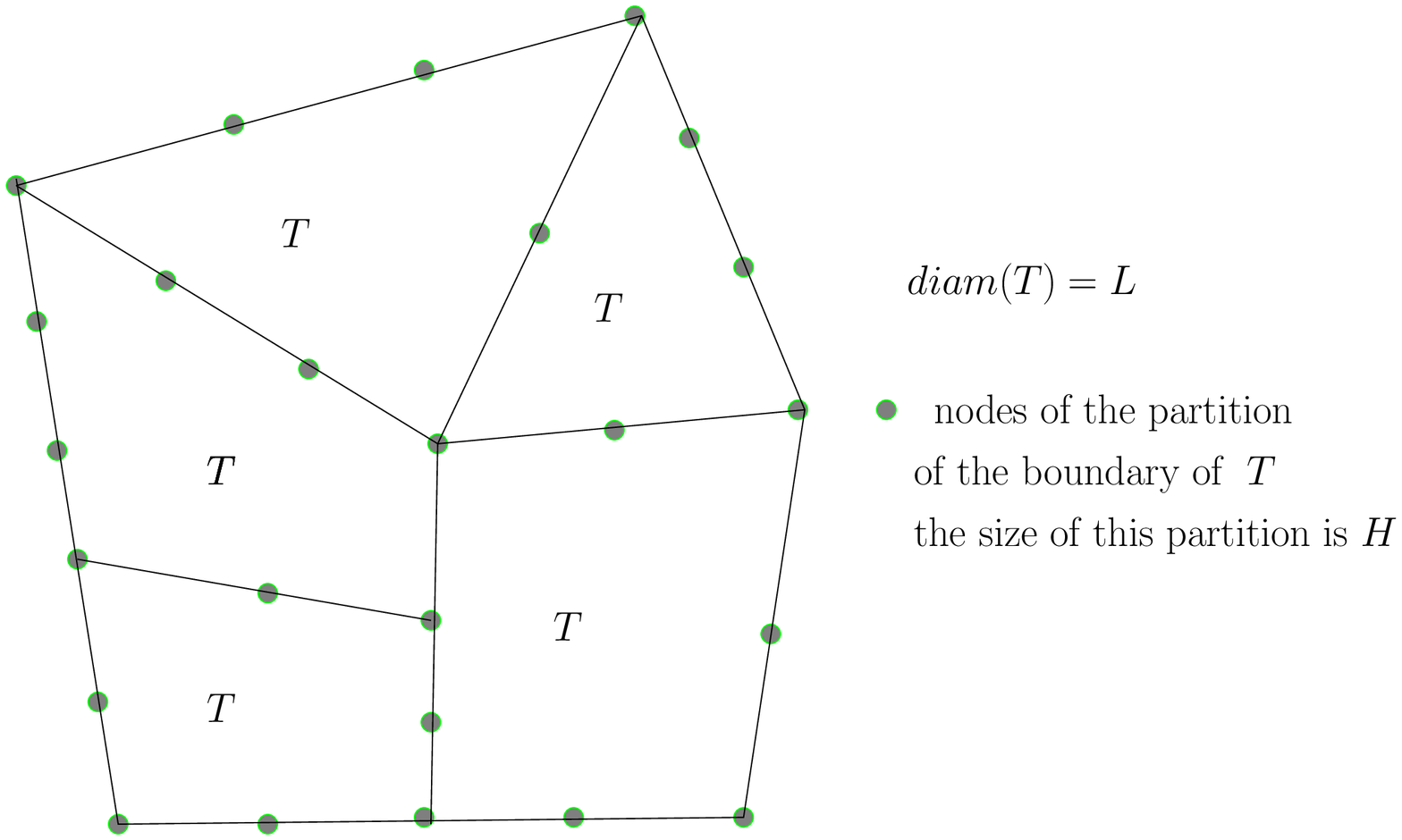}\hfill
 \includegraphics[ width=.4\textwidth, clip]{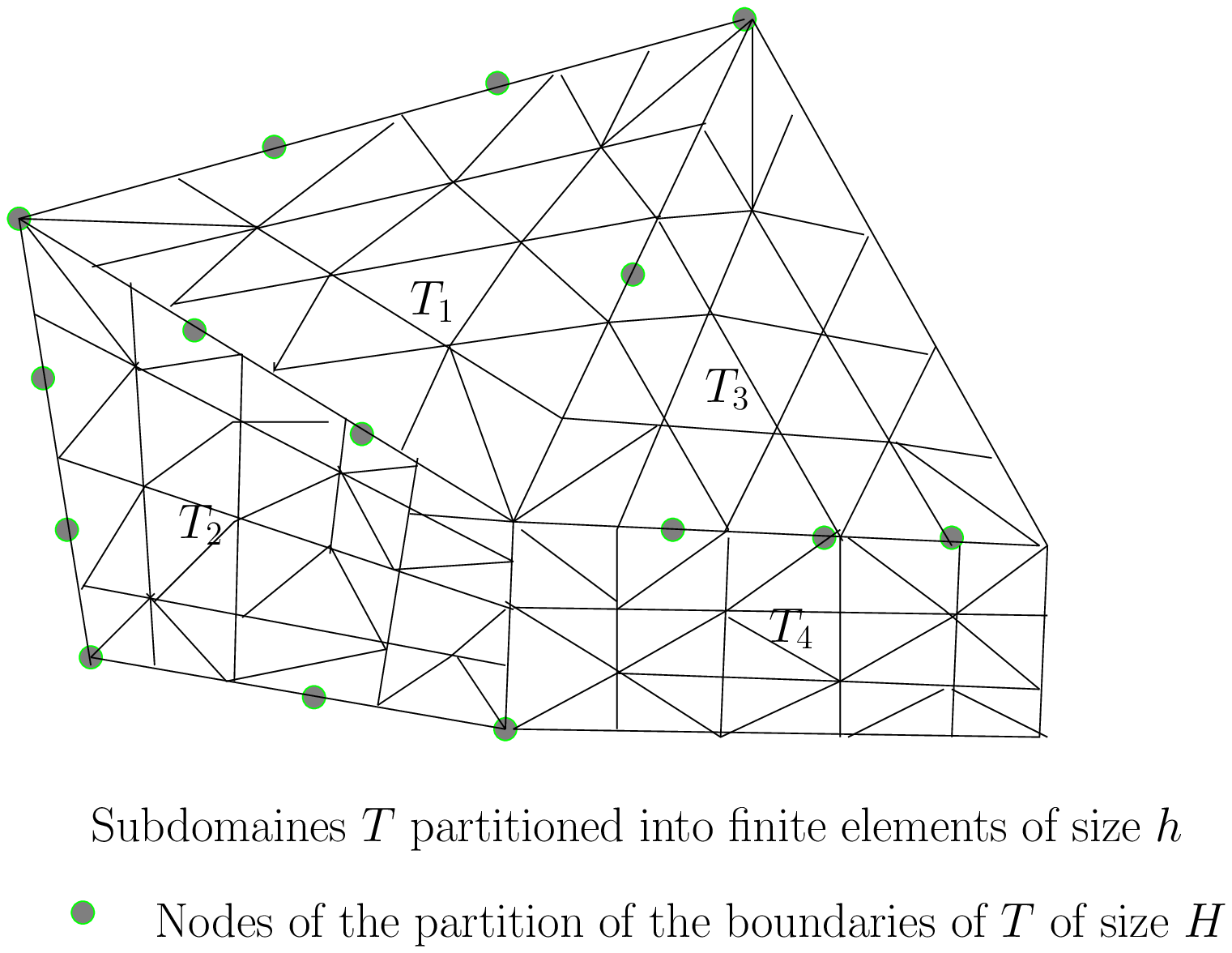}\hfill
}
\centering
\vskip5mm
\caption{Partition of $\Omega$: (left) on the boundaries of the subdomains $T$ of size $L$ an additional mesh of size $H$ is shown;
(right) on each   subdomain $T$  a fine mesh is introduced}
\label{fig:partition}
\end{figure}

Thus, we have three scales: (1) $L$ -- the maximum size of the of the subdomains $T \in \mathcal{P} $, (2) $H$ -- the size of the 
partition of the boundaries of $T \in \mathcal{P} $, and finally (3) the scale of the fine-grid mesh -- the maximum 
diameter $h$ of the finite elements introduced in each subdomain $T \in \mathcal{P} $. 
In this paper we shall assume that the $diam(\Om)=1$ and  $ 0 < h << H \le L \le 1$.

Below is a summary of the above notation by grouping them  into categories according to the scale they represent:
$$
\begin{array}{lrll}
(a)~~ &  & \hspace{-0.7in} \text{partition of the domain $\Omega$ into subdomains $T$ (scale $L$):}  & \\
  & \mathcal{P} & :=  \text{the set of all subdomains $T$ }& \\
  &  \partial \mathcal{P} & := \cup_{T \in \mathcal{P}} \partial T & \\
(b)~ &  
  & \hspace{-0.7in} \text{partition of the boundaries of subdomains $T$ (scale $H$):}  & \\
  & \mathcal{E}_H(T) & := \text{ the set of all coarse edges/faces of a subdomain }  T \in \mathcal{P} & \\
  & \mathcal{E}_H & := \text{ the partition all edges/faces of the boundaries } \partial T,  \,  T \in  \mathcal{P} & \\
(c) &  & \hspace{-0.7in} \text{partition of each subdomain $T \in \mathcal{P} $ into finite elements (scale $h$):}   & \\
 &   \mathcal{T}_h(T) & :=  \text{fine grid triangulations of a subdomain} ~~T\in \mathcal{P}   &\\
 & \mathcal{E}_h (T) & :=\text{ the set of all edges/faces of the triangulation $\TT$ } & \\
 & \mathcal{E}^0_h (T) & :=\text{ the set of all interior edges/faces of the triangulation $\TT$ } (\equiv \mathcal{E}_h (T) \cap T)& \\
 & \partial \mathcal{T}_T & := \cup_{K \in \mathcal{T}_T} \partial K & \\
%
(d) &  & \hspace{-0.7in} \text{globally defined meshes on $\Omega$:}  & \\
 & \mathcal{T}_h & :=  \cup_{T \in \mathcal{P}} \TT  &  \\
 & \partial \mathcal{T}_h & := \cup_{K \in \mathcal{T}_h} \partial K & \\
 & \mathcal{E}_h & := \cup_{T\in \mathcal{P}} \mathcal{E}_h(T) & \\
 & \mathcal{E}^0_h & :=\text{the set of all }  F \in \mathcal{E}_h: ~ F \in \mathcal{E}_h, 
                           ~ F \text{  does not intersect $\partial T $ for any} ~ T\in \TT  & \\
  & \mathcal{E}_{h,H} & := \mathcal{E}^0_h \cup \mathcal{E}_H.  & \\
\end{array}
$$
Note that the scale $H$ is associated only with the partition of the boundaries of the subdomains $T$
of the partition $ \mathcal{P}$.

\subsection{Multiscale FEM}

The methods we are interested in seek an approximation to $(u, \boldsymbol{q}, u|_{\Eh})$
by the hybridized discontinuous Galerkin finite element method. For this purpose we need finite element spaces 
for these quantities consisting of piece-wise polynomial functions. Namely, we introduce
\begin{alignat*}{1}
W_h:=&\;\{w\in{L}^2(\Oh):\;w|_K\in W(K),\; K\in\Oh\}, 
\\
\boldsymbol{V}_h:=&\;\{\boldsymbol{v}\in\boldsymbol{L}^2(\Oh): \;\boldsymbol{v}|_K\in\boldsymbol{V}(K), \; K\in\Oh\}, 
\\
M_{h,H}:=&M^0_h \oplus M_H,\\
\intertext{where the spaces $M^0_h, M_H$ are defined as}
M^0_h:=&\;\{\mu\in{L}^2(\EhH): \; \text{  for } F\in\Eh^0 ~~ \mu|_{F}\in M_h(F), \;  \text{  and } \; \mu|_{\mathcal{E}_H} = 0\},\\
M_H:=&\;\{\mu\in{L}^2(\EhH):\;\text{  for }  F \in\EH  ~~ \mu|_{F}\in M_H(F),\;\;  \text{  and } \mu|_{{\Eh^0}\cup \partial \Omega} = 0\}.
\end{alignat*}
Now the hybridizable multicsale DG FEM reads as follows: find
$(u_h, \boldsymbol{q}_h,$ $\uhat)$ in the space $W_h \times \Vh \times M_{h,H}$
that satisfies the following weak problem
\begin{subequations}
\label{weak formulation}
 \begin{alignat}{3}
  \label{weak formulation-1}
 \bint{\al \boldsymbol{q}_h}{\boldsymbol{v}} &-\bint{u_h}{\nabla \cdot \boldsymbol{v}} &
      & + \bintEh{\uhat}{ \boldsymbol{v} \cdot \boldsymbol{n}} && = 0 \quad \forall \boldsymbol{v} \in \Vh, \\
\label{weak formulation-2}
-\bint{\boldsymbol{q}_h}{\nabla w} & && + \bintEh{\qhat \cdot \n}{w} && = \bint{f}{w} \quad \forall w \in W_h,\\
\label{weak formulation-3}
& && \quad \, \langle{\qhat \cdot \n},{\mu}\rangle_{\partial\Oh} && = 0 \quad \forall \mu \in M_{h,H}, \\
\label{weak formulation-4}
& && \quad \widehat{u}_{h,H} && =  0 \quad \text{on $\partial \Omega$}.
 \end{alignat}
\end{subequations}
Since by the requirement $\uhat \in M_{h,H}$  the last equation is trivially satisfied and therefore
it is redundant. However, we
 prefer to have it written explicitly for later use in the error analysis.

For $\mathcal{T}= \mathcal{T}_h, \TT$, we write 
$(\eta\;,\;\zeta)_{\mathcal{T}} := \sum_{K \in \mathcal{T}} (\eta, \zeta)_K$,  
where $(\eta,\zeta)_D$ denotes the integral of $\eta\zeta$ over 
the domain $D \subset \mathbb{R}^n$. We also write
$\langle{\eta}\; , \; {\zeta}\rangle_{\partial \mathcal{T}}:= \sum_{K \in \mathcal{T}} \langle \eta \,,\,\zeta \rangle_{\partial K},$
where $\langle \eta \,,\,\zeta \rangle_{\partial D}$ 
denotes the integral of $\eta \zeta$ over the boundary of the
domain $D \subset \mathbb{R}^{n-1}$. The definition of the method
is completed with the definition of the normal component of the numerical trace:
\begin{equation}
\label{trace-q}
 \qhat\cdot \n = \boldsymbol{q}_h \cdot \n + \alphaa (u_h - \uhat) \quad \text{ on $\partial \Oh$}.
\end{equation}

On each $K \in \mathcal{T}_h$, the stabilization parameter $\alphaa$ 
is non-negative constant on each $F \in \partial K$ and we 
assume that $\alphaa > 0$ on at least one face $F^* \in \partial K$. 
By taking particular choices of the local spaces $\boldsymbol{V}(K)$, $W(K)$ and $M_h(F), M_H(F)$,
and the {\em linear local stabilization} {operator} $\alphaa$, various mixed ($\alphaa=0$) 
and HDG ($\alphaa\neq0$) methods
are obtained. For a number of  such choices we refer to \cite{CockburnGL2009,CockburnQiuShi2012}.
We note that on each fine element $K \in \mathcal{T}_h$ the local spaces 
$W(K) \times \boldsymbol{V}(K) \times M_h(F)$ can be any set of the spaces presented 
in  \cite[Tables 1 -- 9]{CockburnQiuShi2012}.
It could be any classical mixed elements or the HDG elements defined on different triangulations. 
In Table \ref{table:simplex} we give examples of local spaces for the classical mixed element 
and HDG element defined on a simplex.

\begin{table}[h!]
 \caption{Possible choices for the finite element spaces  for $K$ a simplex.}
\centering
\begin{tabular}{c c c c c}
\hline
\noalign{\smallskip} 
method &         $\boldsymbol{V}(K)$         &     $W(K)$         &      $M_h(F), \; F \in \partial K$    &     $M_H(F), \; F \in \mathcal{E}_H$ \\
\noalign{\smallskip}
\hline\hline
\noalign{\smallskip}
${\mathbf{BDFM}_{k+1}}$ &\hskip-1truecm $\{\boldsymbol{q}\in \boldsymbol{P}^{k+1}(K):$
                      & $P^{k}(K)$ & $P^k(F)$ & $P^{l}(F)$ \\
                      & $\boldsymbol{q} \cdot \n|_{\partial K} \in P^{k}(F), \; \forall F \in \partial K \} $&&&\\
${\mathbf{RT}_k}$ & $\boldsymbol{P}^k(K) \oplus \boldsymbol{x} \widetilde{P}^k(K)$ & $P^k(K)$ & $P^k(F)$ & $P^{l}(F)$ \\
${\mathbf{HDG}_k}$ & $\boldsymbol{P}^k(K)$ & $P^k(K)$ & $P^k(F)$ & $P^{l}(F)$ 
\\
\noalign{\smallskip} 
\hline 
\end{tabular}
\label{table:simplex}
\end{table}


One feature of our formulation is that the choice of the space $M_H(F)$ is totally free. In this paper,  we will consider 
two different choices. The first choice is
the space of piece-wise polynomials defined in \eqref{Lagrange}, while the second is the space 
uses multiscale functions defined in \eqref{Lagrange-MS}. In general it can consist of any function spaces.

\subsection{The upscaled structure of the method}\label{upscaledFEM}
The main feature of this method is that it could be implemented in such a way that we
need to solve certain global system on the coarse mesh $\mathcal{T}_H$ only.
To show this possibility, we  split \eqref{weak formulation-3} into two equations by testing separately 
 with  $\mu \in M^0_h $ and $\mu \in M_H$ so that
\begin{equation}\label{split-system}
\langle{\qhat \cdot \n},{\mu}\rangle_{\partial\Oh}  = 0 ~~~\forall \mu \in M^0_h \quad \text{and} \quad
\langle{\qhat \cdot \n},{\mu}\rangle_{\partial\Oh} = \langle{\qhat \cdot \n},{\mu}\rangle_{\partial\mathcal{T}_H}  = 0
~~~\forall \mu \in M_H.
\end{equation} 
Here $\langle \eta\;,\;\zeta\rangle_{\partial\mathcal{T}_H}:= \sum^n_{T \in \TT} \int_{\partial T} \eta \, \zeta ds$.
On any subdomain $T$, given the boundary data of $\widehat{u}_{h,H} = \xi_H$ for 
$\xi_H \in M_H(F), F \in  \mathcal{E}_{H}(T)$, 
we can solve for $(\boldsymbol{q}_h, u_h, \widehat{u}_{h,H})|_T$ by 
restricting the equations \eqref{weak formulation-1}--\eqref{weak formulation-3}
on this particular $T$:
\begin{alignat*}{3}
  \bintOhi{\al \boldsymbol{q}_h}{\boldsymbol{v}} &-\bintOhi{u_h}{\nabla \cdot \boldsymbol{v}} 
&& + \bintET{\uhat}{ \boldsymbol{v} \cdot \boldsymbol{n}} && = 0, \\
-\bintOhi{\boldsymbol{q}_h}{\nabla w} & &&  +\bintET{\qhat \cdot \n}{w} && = \bintOhi{f}{w}, \\
& && \quad \, \langle{\qhat \cdot \n},{\mu}\rangle_{\partial \TT} && = 0,\\
& && \quad \,\qhat\cdot \n  &&= \boldsymbol{q}_h \cdot \n + \alphaa (u_h - \uhat) ~~ \text{ on} ~~\partial \TT\\
& && \; \quad \uhat &&= \xi_H \quad \text{on $\partial T$,}
\end{alignat*}
for all $(w, \boldsymbol{v}, \mu) \in W_h|_{T} \times \boldsymbol{V}_h|_{T} \times M^0_h|_{\mathcal{E}^0_h(T)}$. 
In fact the above local system is
the regular HDG methods defined on $T$.  From \cite{CockburnQiuShi2012} 
we already know that this system is stable.
Hence, this  HDG solver defines a global {\it affine}  mapping from $M_H$ to $W_h \times \boldsymbol{V}_h \times M^0_h$.  
The solution can be further split into two parts, namely, 
%
$$ 
(\boldsymbol{q}_h, u_h, \widehat{u}_{h,H}) =(\boldsymbol{q}_h(f), u_h(f), 
\widehat{u}_{h,H}(f)) 
+(\boldsymbol{q}_h(\xi_H), u_h(\xi_H), \widehat{u}_{h,H}(\xi_H))
$$ 
where $ (\boldsymbol{q}_h(f), u_h(f), \widehat{u}_{h,H}(f)) $ satisfies
\begin{alignat*}{3}
 \bintOhi{\al \boldsymbol{q}_h(f)}{\boldsymbol{v}} &-\bintOhi{u_h(f)}{\nabla \cdot \boldsymbol{v}} && 
      + \bintET{\uhat(f)}{ \boldsymbol{v} \cdot \boldsymbol{n}} && = 0, \\
-\bintOhi{\boldsymbol{q}_h(f)}{\nabla w} & && + \bintET{\qhat(f) \cdot \n}{w} && = \bintOhi{f}{w}, \\
& && \quad \, \langle{\qhat(f) \cdot \n},{\mu}\rangle_{\partial \TT} && = 0,\\
& && \; \quad \uhat &&= 0 \quad \text{on $\partial T$,}
\end{alignat*}
for all $(w, \boldsymbol{v}, \mu) \in W_h|_{T} \times \boldsymbol{V}_h|_{T} \times M^0_h|_{\mathcal{E}^0_h(T) }$
and $ (\boldsymbol{q}_h(\xi_H), u_h(\xi_H), \widehat{u}_{h,H}(\xi_H))$ satisfies
\begin{alignat*}{3}
 \bintOhi{\al \boldsymbol{q}_h(\xi_H)}{\boldsymbol{v}} &
-\bintOhi{u_h(\xi_H)}{\nabla \cdot \boldsymbol{v}} && 
+ \bintET{\uhat(\xi_H)}{ \boldsymbol{v} \cdot \boldsymbol{n}} && = 0, \\
-\bintOhi{\boldsymbol{q}_h(\xi_H)}{\nabla w} & && + \bintET{\qhat(\xi_H) \cdot \n}{w} && = 0, \\
& && \quad \, \langle{\qhat \cdot \n},{\mu}\rangle_{\partial \TT} && = 0,\\
& && \; \quad \uhat(\xi_H) &&= \xi_H \quad \text{on $\partial T$,}
\end{alignat*}
for all $(w, \boldsymbol{v}, \mu) \in W_h|_{T} \times \boldsymbol{V}_h|_{T} \times M^0_h|_{\mathcal{E}^0_h(T)}$.

Then the second equation \eqref{split-system}
reduces to 
\begin{equation}\label{upscaled equation}
 a(\xi_H, \mu) = l(\mu) 
\qquad \text{for all $\mu \in M_H$,} 
\end{equation}
where the bilinear form $a(\xi_H, \mu): M_H \times M_H \to R$ 
and the linear form $ l(\mu): M_H \to R$ are defined as
\begin{equation}\label{form-a}
a(\xi_H, \mu):= \bintEH{\qhat(\xi_H) \cdot \n}{\mu}\quad \text{and} \quad  l(\mu):=a(f,\mu)=\bintEH{\qhat(f) \cdot \n}{\mu} . 
\end{equation} 

\begin{remark}
The same procedure can be applied also for the case of non-homogeneous data $u=g$ on $\partial \Omega$.
However, the presentation of this case is much more cumbersome. In order to simplify the notations and to highlight 
the main features of this method we have assumed  homogeneous Dirichlet boundary data.
\end{remark}

\subsection{Existence of the solution of the FEM} \label{solvability}

The framework is general in terms of flexibility in  the choice of the local spaces. 
However, in order to ensure the solvability of the system, we need some 
assumptions. 
%
\begin{assumption}\label{local_lifting_0}
For any $K \in \mathcal{T}_h$,  $F^*$  an arbitrary face of $K$, 
and $\mu \in M_h(F), F \in \partial K$, 
there exists
a element $\boldsymbol{Z} \in \boldsymbol{V}(K)$ such that
\begin{alignat*}{2}
(\boldsymbol{Z}, \nabla w) & = 0, \quad && \text{for all $w \in W(K)$,} \\ 
\boldsymbol{Z} \cdot \n|_F &= \mu, \quad && \text{for all $F \in \partial K \backslash F^*$.} 
\end{alignat*} 
\end{assumption}

This assumption is trivially satisfied by all classical  mixed finite elements, e.g.
\textbf{RT, BDM, BDDF}, etc. For these elements one can simply define 
$\boldsymbol{Z} = \boldsymbol{\Pi}_h  \boldsymbol{Q}$, where
$\boldsymbol{Q}$ is any solution of the problem:
$$ 
 \nabla \cdot \boldsymbol{Q} = 0 \quad \text{in $K$} \quad \mbox{and} \quad
\boldsymbol{Q} \cdot \n = \mu \quad \text{on $\partial K$},
$$ 
where $\boldsymbol{\Pi}_h$ is the Fortin projection to the mixed elements (see, e.g. \cite{Brezzi_Fortin_book}). 
For the case of simplex triangulations and HDG elements, 
we refer the reader to  \cite[Lemma 3.2]{CockburnDongGuzman2008}. 
The proof for other HDG elements are very similar to the case of simplicial elements considered 
in \cite{CockburnDongGuzman2008}.

Further, we need an assumption on the stabilization parameter $\alphaa$:
\begin{assumption}\label{tau-parameter}
On each $F_H \in \mathcal{T}_{E}$, for any $T$ adjacent to $F_H$, i.e. $\bar{T} \cap F_H \ne \emptyset$, 
there exists at least one element $K \in \mathcal{T}_T$ adjacent to $F_H$, 
such that the stabilization operator $\alphaa > 0$ on $F^* = F_H \cap \partial K$.
\end{assumption}

We are now ready to show the solvability of the method.
\begin{theorem}\label{stability}
 Let Assumptions \ref{local_lifting_0}  and  \ref{tau-parameter} be satisfied.
Then for any $f$, the 
system \eqref{weak formulation} has a unique solution. 
\end{theorem}

\begin{proof}
Notice that the system \eqref{weak formulation} is a square system. It suffices to show 
that the homogeneous system has only the trivial solution. 
From \eqref{weak formulation-4} we see that $\uhat = 0$ on $\partial \Omega$.
Now assume that $( u_h, \boldsymbol{q}_h, \uhat )$ is any solution of \eqref{weak formulation}.
Setting $(w, \boldsymbol{v}, \mu) = (u_h, \boldsymbol{q}_h, \uhat)$ in 
\eqref{weak formulation-1}-\eqref{weak formulation-3} and adding all equations, we get 
after some algebraic manipulation,
\[
\bint{\al \boldsymbol{q}_h}{\boldsymbol{q}_h} - 
\bintEh{\boldsymbol{q}_h \cdot \n - \qhat \cdot \n}{u_h - \uhat} = 0.
\]
 By the definition of the numerical traces \eqref{trace-q}, we have
\[
\bint{\al \boldsymbol{q}_h}{\boldsymbol{q}_h} + \bintEh{\alphaa(u_h - \uhat)}{u_h - \uhat} = 0
\]
and since $\alphaa \ge 0$ we get
\begin{equation}\label{stability_1}
\boldsymbol{q}_h = 0, \qquad \alphaa(u_h - \uhat) = 0
\end{equation}
and \eqref{weak formulation-1} becomes
\[
-\bint{u_h}{\nabla \cdot \boldsymbol{v}} + \bintEh{\uhat}{\boldsymbol{v} \cdot \n} = 0, 
       \quad \text{for all $\boldsymbol{v} \in \boldsymbol{V}_h$}.
\]
Now we take this over an element $K$
and after integration by parts 
we get
\begin{equation}\label{boundary-term}
\bintK{\nabla u_h}{\boldsymbol{v}} + \bintEK{\uhat - u_h}{\boldsymbol{v} \cdot \n} = 0, 
          \quad \text{for all $\boldsymbol{v} \in \boldsymbol{V}(K)$.}
\end{equation}
Since $\alphaa > 0$ on $F^* \in \partial K$ than the second equality
\eqref{stability_1} implies that
$
u_h - \uhat =0 \quad \mbox{on} \quad F^*.
$
Next, by Assumption \ref{local_lifting_0},  there is
$\boldsymbol{v} \in \boldsymbol{V}(K)$ such that
\begin{alignat*}{2}
(\boldsymbol{v}, \nabla w) & = 0, \quad && \text{for all $w \in W(K)$,}\\
\boldsymbol{v} \cdot \n|_F &= P^h_{\partial} \uhat - u_h, \quad && 
     \text{for all $F \in \partial K \backslash F^*$.} 
\end{alignat*} 
where for $K \in \mathcal{T}_h$, $P^h_{\partial}: \, L^2(F)  \to M_h(F)$ is the local $L^2-$orthogonal projection  
onto $M_h(F)$, for all $F \in \partial K$.
Inserting such $\boldsymbol{v}$ in \eqref{boundary-term}, we get 
\begin{alignat*}{1}
0  = \bintK{\nabla u_h}{\boldsymbol{v}} + \bintEK{\uhat - u_h}{\boldsymbol{v} \cdot \n} 
  = \langle P^h_{\partial} \uhat - u_h \, , \, P^h_{\partial} \uhat - u_h \rangle_{\partial K \backslash F^*}.
\end{alignat*}
This implies that $P^h_{\partial} \uhat - u_h = 0$ on $\partial K \backslash F^*$. 
Since on $F^*$, $P^h_{\partial} \uhat - u_h = P^h_{\partial} (\uhat - u_h) = 0$
we get 
\begin{equation}\label{stability_2}
P^h_{\partial} \uhat - u_h = 0 \quad \text{on $\partial{K}$, \; for all $K \in \mathcal{T}_h$}.
\end{equation}

Moreover, this means that $\bintK{\nabla u_h}{\boldsymbol{v}} = 0$ for all $\boldsymbol{v} \in \boldsymbol{V}(K)$.
Taking $\boldsymbol{v} = \nabla u_h$, we have $u_h$ is piecewise constant on each $K \in \mathcal{T}_h$.
The above equation shows that $u_h = P^h_{\partial} \uhat$ on $\partial K$. 
On each $T$, $\mathcal{T}_T$ is a conforming triangulation, so
this implies that for any interior face $F \in \mathcal{E}^0_{h}$ shared by two neighboring 
elements $K^+, K^-$, the local spaces satisfy $M_h(F^+) = M_h(F^-)$ and hence
$P^h_\partial {\uhat}$ coincides from both sides. This implies that in fact $u_h= C_T$ in each subdomain $T$ 
and $\uhat|_{\mathcal{E}^0_h \cap T} = C_T$.

Next, on each $F_H \in \mathcal{E}_{H}$, we assume $F_H \subset \bar{T_1} \cap \bar{T_2}$, if 
$F_H \subset \partial \Omega$ then $F_H \subset \partial T_1$. By  Assumption \ref{tau-parameter} 
there exists $K_1 \in \mathcal{T}_{T_1}, K_2 \in \mathcal{T}_{T_2}$ adjacent
to $F_H$ such that $\tau > 0$ on $F_i = \partial K_i \cap F_H$, $i=1,2$. By \eqref{stability_1}, we have 
\[
\uhat - u_h = 0 \quad \text{on $F_i,\quad i = 1, 2$.}
\]
This implies that $\uhat|_{F_H} = C_{T_1} = C_{T_2}$. Hence we have $C_{T}=C $ for all $T$, 
which means that
$u_h = C$ over the domain $\Omega$ and $\uhat|_{\mathcal{E}_{h}} = C$. 
Finally, by the fact that $\uhat = 0$ on $\partial \Omega$, we must have $u_h = \uhat=C = 0$ and
this completes the proof.  
\end{proof}

In \cite{Arbogast_PWY_07}, in order to ensure the solvability of the mortar methods, the key assumption (roughly speaking) is 
that on $\mathcal{E}$ the fine scale space $M_h$ should be rich enough comparing with the coarse scale space $M_H$.
In this paper, since the stabilization is achieved by the parameter $\alphaa$ we prove stability under  the assumption  
that on each $F_H \in \mathcal{E}_H$ the parameter $\alphaa$ is strictly 
positive on some portion of $F_H$. We do not need any conditions between the local spaces $M_h(F_h)$ and $M_H(F_H)$.

\section{Error Analysis}\label{errors}
In this section we derive error estimates for the proposed above method. We would like
stress on two important points of this method.  First,
in the most general case we have three different scales in our partitioning. The error estimates
should reflect this generality of the setting.
Second, upon different choices of the spaces and the stabilization strategy (i.e.
the choice of the parameter $\alphaa$) we can get different convergence rates. 
For example, to obtain error estimates of optimal order we have to make some additional
assumptions. All these are discussed in this section. For the sake of simplicity, we assume that the nonzero stabilization 
parameter $\tau$ is constant on all element $K \in \Oh.$ In this section and the one follows, we only consider the method 
with the coarse space defined by polynomials, that is:
\begin{equation}\label{Lagrange}
M_H(F) = P^l(F), \quad \text{for all} \quad F \in \mathcal{E}_H.
\end{equation}

\subsection{Preliminary Results}\label{prelims}

We present the main results in this section. In order to carry out a priori error estimates, 
we need some additional assumptions on the scheme.
The first assumption is identical to \emph{Assumption A} in \cite{CockburnQiuShi2012}, in order to be self-consistent, we still present it here:
\begin{assumption}\label{assum-inclusions}
The local spaces satisfy the following inclusion property:
\begin{subequations}\label{inclusion_prop}
\begin{alignat}{1}
\label{W_in_M} W(K)|_{F} \subset M_h(F) \quad & \text{for all $F \in \partial K$}, \\
\label{V_in_M} \boldsymbol{V}(K) \cdot \n|_F \subset M_h(F) \quad & \text{for all $F \in \partial K$}.
\end{alignat}
\end{subequations}
On each element $K \in \mathcal{E}_h$, there exist  local projection operators
%
$$
\Piw:~ H^1(K) \to W(K)  \quad \text{and} \quad \Piv:~ \boldsymbol{H}_{div}(K) \to \boldsymbol{V}(K)
$$
associated with the spaces $W(K)$, $\boldsymbol{V}(K)$, $M_h(F)$ defined by:
\begin{subequations}\label{HDG projection}
 \begin{alignat}{2}\label{HDG projection-1}
  (u,w)_K &= (\Piw u, w)_K \quad && \text{for all $w \in \nabla \cdot \boldsymbol{V}(K)$,}\\
\label{HDG projection-2}
 (\boldsymbol{q}, \boldsymbol{v})_K &= (\Piv \boldsymbol{q}, \boldsymbol{v})_K  
               \quad && \text{for all $\boldsymbol{v} \in \nabla W(K)$,}\\
\label{HDG projection-3}
\langle \boldsymbol{q} \cdot \n + \alphaa u \, , \, \mu \rangle_F 
              &= \langle \Piv \boldsymbol{q} \cdot \n + \alphaa \Piw u \, , \, \mu \rangle_F 
            \quad && \text{for all $\mu \in M_h(F), F \in \partial K$.} 
 \end{alignat}
\end{subequations}
\end{assumption}

\begin{assumption}\label{single-faced HDG}
On each fine element $K$, the stabilization operator $\tau$ is strictly positive on only one face $F \in \partial K$. 
\end{assumption}

\begin{assumption}\label{tau_positive_EH}
For any element $K$ adjacent to the skeleton  $\partial \mathcal{P} $ 
on a face $F$, shared by $K$ and $\partial \mathcal{P}$,
$\tau$ is strictly positive, i.e.  $\tau|_F > 0$.
\end{assumption}

The above suggested local spaces $W(K) \times \boldsymbol{V}(K) \times M_h(F)$
or any set of local spaces presented in \cite{CockburnQiuShi2012} 
satisfy Assumption \ref{assum-inclusions}. 
Moreover, 
assumptions \ref{single-faced HDG} and \ref{tau_positive_EH} are the key to obtain optimal approximation results. 
In fact, without these two assumptions, we can still get some error estimates. However, the result will have a term with 
negative power of $h$ which is not 
desirable since $h$ is the finest scale. We will discuss this issue at the end of Section 4.2. 

As a consequence of Assumption \ref{single-faced HDG} and \ref{tau_positive_EH}, 
the triangulation of each subdomain has 
to satisfy the requirement that {\it 
each fine scale finite element $K \in \mathcal{T}_h$ can share at most one face with the 
coarse skeleton $\mathcal{E}_H$.}
This requirement implies that we need to put at least two fine elements to fill a corner of any subdomain. 
This suggests that we should use triangular (2D) or tetrahedral (3D) elements. In what follows, we restrict the choice 
of local spaces to be in Table \ref{table:simplex}. Notice that here we exclude the famous $\mathbf{BDM}_k$ space 
from the table. Roughly speaking, the reason is that in the case of $\mathbf{BDM}_k$ element, the local space $W(K) = P^{k-1}(K)$  
is too small to provide a key property for the optimality of the error bound, see Lemma \ref{H_div_conforming}. 

In \cite{CockburnQiuShi2012} it has been shown that for any 
$(u,\boldsymbol{q}) \in H^1(K) \times \boldsymbol{H}_{div}(K)$, 
the projection $(\Piw u, \Piv \boldsymbol{q}) \in W(K) \times \boldsymbol{V}(K)$ 
exists and is unique. 
Moreover, for all elements listed in Table \ref{table:simplex}, the projection has 
the following approximation property: 
\begin{lemma}\label{projection approximation}
If the local spaces $\boldsymbol{V}(K), W(K)$ are mixed element spaces $\mathbf{RT}_k$ or $\mathbf{BDFM}_{k+1}$,  then 
\begin{align*}
\|\boldsymbol{q} - \Piv \boldsymbol{q}\|_K \le  C h^s(\|\boldsymbol{q}\|_{s,K} + \tau \|u\|_{s,K}) \quad \mbox{and} \quad
\|u-\Piw u\|_K \le  C h^s \|u\|_{s,K}
\end{align*}
and if the local spaces $\boldsymbol{V}(K), W(K)$ are  $\mathbf{HDG}_k$ spaces, then
\begin{align*}
\|\boldsymbol{q} - \Piv \boldsymbol{q}\|_K \le & C h^s(\|\boldsymbol{q}\|_{s,K}) \quad  \mbox{and} \quad 
\|u-\Piw u\|_K \le  C h^s (\|u\|_{s,K} + \tau^{-1} \|\boldsymbol{q}\|_{s,K})
\end{align*}
for all $ 1 \le s \le k+1$.
\end{lemma}

Further in our analysis we shall need  some auxiliary projections and their properties:
\begin{equation}\label{projections}
\begin{array}{lllll}
   P^H_{\partial}: & L^2(F) & \to & M_H (F),  \quad \langle P^H_{\partial} u, \mu \rangle_F =  \langle u, \mu \rangle_F           
                          \quad \forall  F \in \mathcal{E}_H, & \\
   P^h_{\partial}: & L^2(F) & \to &  M_h(F),  \quad \langle P^h_{\partial} u, \mu \rangle_F =  \langle u, \mu \rangle_F 
                          \quad \forall F \in \mathcal{E}^0_h, \\
   P_M: & L^2(\mathcal{E}_{h,H}) & \to & M_{h,H},  ~~\text{  with   }~~ P_M =\left \{ \begin{array}{ll}
                                                                                                                 P^{H}_{\partial} & \text{  on  } \mathcal{E}_H, \\
                                                                                                                 P^h_{\partial}     & \text{  on $\mathcal{E}^0_h$}, 
                                                                                                           \end{array}  \right .\\
\mathcal{I}^0_H:  &  C(\Omega)   & \to  &  M^c_H 
   ~~\text{ with  ~~$M^c_H \subset M_H$, } 
\end{array}
\end{equation}
where  $M^c_H$ is the subset of $M_H$ of continuous functions and $ \mathcal{I}^0_H$  the Lagrange (nodal)  interpolation operator.

From the last equation \eqref{HDG projection-3} and the definitions \eqref{projections} of the projection operators, 
we have
\begin{equation}\label{bdry_identity}
P^{\partial}_h(\boldsymbol{q} \cdot \n) + \tau P^{\partial}_h u = \Piv \boldsymbol{q} \cdot \n + \tau \Piw u, \quad \text{for all $F \in \partial \mathcal{T}_h$}.
\end{equation}

In the analysis, we will need the following useful approximation properties of the projections $ P^h_{\partial}$, $ P^h_{\partial}$
and the interpolation operator $ \mathcal{I}^0_H$:
\begin{lemma}\label{bdy_to_interior}
For any $T \in \mathcal{T}$ and any smooth  enough function $u$ we have
\begin{align}
\|u - P^h_{\partial} u\|_{\partial T} &\le C L^{-\frac{1}{2}} h^s \|u\|_{s + 1, T},  
\quad && 0 \le s \le k+1, \label{Phd} \\
\|u - P^H_{\partial} u\|_{\partial T} & \le C L^{-\frac{1}{2}} H^t \|u\|_{t + 1, T}, && 0 \le t \le l + 1,  \label{PHd}\\
\|u - \mathcal{I}^0_H u\|_{\frac{1}{2}, \partial T} 
             &\le C L^{-\frac{1}{2}} H^{t - \frac{1}{2}} \|u\|_{t + 1, T} && 0 \le t \le l+1.  \label{IH0}
\end{align}
\end{lemma}
Here the constant $C$ solely depends on the shape of the domain $T$ but not if its size.

\begin{remark}
The regularity assumptions  $H^{s+1} (H^{t+1})$
can be weakened  to  $H^{s+\frac{1}{2} + \epsilon} (H^{t+ \frac{1}{2} + \epsilon})$ for any $\epsilon > 0$ without
reducing the approximation order. However, the above estimates make the presentation more transparent and shorter.
\end{remark}

\begin{proof}(of Lemma \ref{bdy_to_interior})
First we note the following standard estimates for the error 
on any edge/face $F \subset \partial T$,  see \cite{Ciarlet1978}:
\begin{subequations}\label{standard_estimates}
\begin{alignat}{2}
\label{standard_1}
\|u - P^h_{\partial} u\|_F &\le C h^s |u|_{s, F}, \quad &\mbox{$s$ integer,} & \quad 0 \le s \le k+1, \\
\label{standard_2}
\|u - P^H_{\partial} u \|_F &\le C H^t |u|_{t, F},  &\mbox{$t$ integer,}& \quad 0 \le t \le l+1, \\
\label{standard_3}
\|(I - P^h_{\partial})(\boldsymbol{q} \cdot \n)\|_F &\le C h^s |\boldsymbol{q} \cdot \n |_{s,F}, &\mbox{$s$ integer,}& \quad 0 \le s \le k+1, \\
\label{standard_5}
\|u - \mathcal{I}_H^0 u\|_{t, \partial T} &\le C H^{s-t}|u|_{s, \partial T}, & \, \, \mbox{$s,t$ integer,}& \, \, 1 < s \le l+1, \; 0\le t \le 1.
\end{alignat}
\end{subequations}
All three inequalities can be obtained by a similar scaling argument. 
Here we only present the proof of the first one of them. Assume $F$ is one of 
the faces of the element $T \in \mathcal{T}$. By \eqref{standard_1}, we have
\begin{align*}
\|u - P^h_{\partial} u \|_{\partial T} &\le C h^s |u|_{s,\partial T} \\ 
& \le C h^s (L^{-\frac{1}{2}} |u|_{s,T} + L^{\frac{1}{2}} |u|_{s+1,T}) \quad \text{by the trace inequality \eqref{eq:trace},}\\
& \le C L^{-\frac{1}{2}} h^s \|u\|_{s+1, T},
\end{align*}
for all  integer $0 \le s \le k+1$. The case of $s$ non integer follows by interpolation and the
other two are proven in a similar way. We note that the factor  $L^{-\frac12}$  related to the scale of the 
subdomains $T$. If the size of $T$ is $O(1)$ then these estimates are well  known.
\hfill
\end{proof}

\begin{remark}
Note that this projections $\Piw$ and $\Piv$ are connected through  the boundary equation
\eqref{HDG projection-3}. Of course, for $\boldsymbol{H}_{div}$-conforming finite element spaces, we can 
take $\alphaa=0$ and these two projections coincide with those of the mixed FEM. 
In particular, $\Piv$ is well defined,  see, e.g. \cite[Section III.3.3]{Brezzi_Fortin_book}.
\end{remark}

\subsection{Main Result}\label{ssec:main}

We are now ready to state two main results for the methods, which proofs will postponed until Section \ref{proofs}. First
we present the estimate for the vector variable  $ \boldsymbol{q}$ the the weighted norm
$$
\|w\|^2_{\alpha, \Omega} = \int_{\Omega} \alpha |w|^2 dx.
$$
\begin{theorem}\label{estimate_q}
Let the local spaces $W(K) \times \boldsymbol{V}(K) \times M_h(F)$ are any from Table \ref{table:simplex} 
and let Assumption \ref{single-faced HDG}, \ref{tau_positive_EH} be satisfied. Then we have
\begin{align*}
\|\boldsymbol{q} - \boldsymbol{q}_h\|_{\alpha, \Omega} & \le C \|\boldsymbol{q} 
       - \Piv \boldsymbol{q}\|_{\alpha, \Omega}+ C H^{t - \frac{1}{2}} L^{-\frac{1}{2}} \|u\|_{t+ 1}\\
& +C \tau H^t L^{-\frac{1}{2}} \|u\|_{t+1} + C\tau h^s L^{-\frac{1}{2}} \|u\|_{s+1} 
     + C \tau^{-\frac{1}{2}} h^s L^{-\frac{1}{2}}\|\boldsymbol{q}\|_{s+1},
\end{align*}
for all $0 \le s \le k+1, \; 0 \le t \le l+1$ with constants $C$ independent 
 of  $u, \boldsymbol{q}, h, H,$ and $ L$. 
\end{theorem}

Next we state the result regarding the error $u - u_h$. It is valid under a typical elliptic regularity
property we state next. Let $(\boldsymbol{\theta},\phi)$ is the solution of the \emph{dual} problem:
\begin{subequations}\label{adjoint}
\begin{alignat}{2}
\label{adjoint-1}
\alpha \boldsymbol{\theta}+\nabla \phi  & = 0 \quad &&\text{in $\Omega$,}\\
\label{adjoint-2}
\nabla \cdot \boldsymbol{\theta} & = e_u &&\text{in $\Omega$,}\\
\label{adjoint-3}
\phi &= 0 && \text{on $\partial \Omega$.}
\end{alignat}
\end{subequations} 
We assume that we have full $H^2-$regularity,
\begin{equation}\label{regularity}
\|\phi\|_{2, \Omega} + \|\boldsymbol{\theta}\|_{1, \Omega} \le C \|e_u\|_{\Omega},
\end{equation}
where $C$ only depends on the domain $\Omega$.

\begin{theorem}\label{estimate_u}
Let the conditions of  Theorem \ref{estimate_q} be satisfied. In addition, assume  full elliptic regularity,  \eqref{regularity},
and the local space $W(K)$ contains piecewise linear functions for each $K \in \mathcal{T}_h$. Then 
for all $1 \le s \le k+1, \; 1 \le t \le l+1$ we have
\begin{align*}
\|u - u_h\|_{\Omega} & \le  \|u - \Piw u\|_{\Omega} \\
&  + C \mathcal{C} \left( \|\boldsymbol{q} - \Piv \boldsymbol{q}\|_{\alpha, \Omega} 
        +(1+\tau H^{\frac{1}{2}}) H^{t - \frac{1}{2}} L^{-\frac{1}{2}} \|u\|_{t+ 1}+\tau h^s L^{-\frac{1}{2}} \|u\|_{s+1} 
          +  \tau^{-\frac{1}{2}} h^s L^{-\frac{1}{2}}\|\boldsymbol{q}\|_{s+1} \right) \\
& + C H^{\frac{3}{2}} L^{-\frac{1}{2}} h^s (\|\boldsymbol{q}\|_{s+1} +\tau \|u\|_{s + 1}) \\
& + C H^t L^{-\frac{1}{2}} \left( H^{\frac{3}{2}} \|\boldsymbol{q}\|_{t+1} + (h^{\frac{1}{2}} + \tau H^{\frac{3}{2}})\|u\|_{t + 1} \right),
\end{align*}
where $\mathcal{C}:= C_{\alpha}h + H + h^{\frac{1}{2}}\tau^{-\frac{1}{2}} + \tau^{\frac{1}{2}} H^{\frac{3}{2}},$
and the constants $C$ are independent of $u, \boldsymbol{q}, h, H, L$. 
\end{theorem}

The above two results are based on a  general framework which utilizes three different 
scales $L, H, h$ and a stabilization parameter $\tau$. 
The richness of the proposed setup gives a  flexibility that allows us to modify the method 
to fit different scenarios. On the other hand, it is hard 
to see the convergent rates of the methods based on this general setup. Now
we discuss the results in details under some practical conditions. 
Here will simply assume the the coefficient $\alpha$ is uniformly bounded.

$\bullet$ {\bf Case 1:  $L = \mathcal{O}(1)$}. Basically, this means that the subdomains $T \in \mathcal{P}$ 
have the same scale as the original domain $\Omega$.  In this case, if we take $\tau = 1$, by the above two theorems and 
Lemma \ref{projection approximation}, we may summarize the order of convergence as follows:
\begin{align*}
\|\boldsymbol{q} - \boldsymbol{q}_h \|_{\Omega} &= \mathcal{O}(H^{l+\frac{1}{2}} + h^{k+1}) \quad \mbox{and} \quad 
\|u - u_h\|_{\Omega} = \mathcal{O}(H^{l+\frac{3}{2}} \max \{1, h^{\frac{1}{2}} H^{-1}\} + h^{k+1}).
\end{align*}  
In this case, our method is very close to the mortar methods introduced in 
\cite{Arbogast_PWY_07}. Indeed, the mortar methods have the following convergence rate:
\begin{align*}
\|\boldsymbol{q} - \boldsymbol{q}_h \|_{\Omega} &= \mathcal{O}(H^{l+\frac{1}{2}} + h^{k+1})  \quad \mbox{and} \quad 
\|u - u_h\|_{\Omega} = \mathcal{O}(H^{l+\frac{3}{2}} + h^{k+1} ).
\end{align*}  
We can see that both methods have exactly the same order of convergence for $\boldsymbol{q}$. For the unknown $u$, 
the HDG methods have an extra term $\max \{1, h^{\frac{1}{2}}H^{-1} \}$. This suggests that HDG method little weaker 
approximation property if $h > H^2$. This is due to the stabilization operator in the formulation. However, the advantage of 
the stabilization is that we don't need any assumption between the spaces $M_h$ and $M_H$. 

If we choose $\tau = H^{-1}$, then the constant $\mathcal{C} = \mathcal{O}(H)$, combining Lemma 
\ref{projection approximation}, Theorem \ref{estimate_q}, \ref{estimate_u}, we obtain the following convergence rate:
\begin{align*}
\|\boldsymbol{q} - \boldsymbol{q}_h \|_{\Omega} &= \mathcal{O}(H^l + h^{k+1} H^{-1})   \quad \mbox{and} \quad 
\|u - u_h\|_{\Omega} = \mathcal{O}(H^{l+1} + h^{k+1}).
\end{align*}
We can see that in this situation, the convergence rates for both $\boldsymbol{q}$ and 
$u$ are slightly degenerated. 

$\bullet$ {\bf Case 2: $H=L$}. From the practical point of view, this assumption suggests that we don't further divide the edges 
of the subdomains $T \in \mathcal{P}$. In this case, we also present the convergence rates by taking $\tau=1, \tau = H^{-1}$, respectively.

For $\tau = 1$, the order of convergence is:
\begin{align*}
\|\boldsymbol{q} - \boldsymbol{q}_h \|_{\Omega} &= \mathcal{O}(H^{l} + h^{k+1} H^{-\frac{1}{2}})  \quad \mbox{and} \quad 
\|u - u_h\|_{\Omega} = \mathcal{O}(H^{l+1} \max\{1, h^{\frac{1}{2}}H^{-1}\} + h^{k+1}).
\end{align*}  

For $\tau = H^{-1}$, the order of convergence is:
\begin{align*}
\|\boldsymbol{q} - \boldsymbol{q}_h \|_{\Omega} &= \mathcal{O}(H^{l-\frac{1}{2}} + h^{k+1} H^{-\frac{3}{2}})  \quad \mbox{and} \quad 
\|u - u_h\| = \mathcal{O}( H^{l+\frac{1}{2}} + h^{k+1} H^{-\frac{1}{2}} ).
\end{align*}  
Similar as in {\bf Case 1}, the convergence rates for both unknowns are worse if we choose $\tau = H^{-1}$.

We can see that if we choose the stabilization parameter $\tau$ inappropriately, 
the numerical solution does not even converge. On the other hand, if all other parameters 
are pre-assigned, we can follow a simple calculation to determine the optimal value 
of $\tau$ for the methods. We will illustrate this strategy with following setting:
we assume that the polynomial degrees $k, l$ are given, $L = H$, 
$h = H^{\alpha} \; (\alpha > 1)$, the local spaces are HDG spaces. 
Then the order of convergence for $\boldsymbol{q}$ solely depends on $\tau$. 
Namely, it can be written as
$
\|\boldsymbol{q} - \boldsymbol{q}_h\| = \mathcal{O} (h^{k+1} + H^l 
   + \tau H^{l+\frac{1}{2}} + \tau h^{k+1}H^{-\frac{1}{2}} + \tau^{-\frac{1}{2}} h^{k+1} H^{-\frac{1}{2}}).
$
Applying the relation $h = H^{\alpha}$ and setting $\tau = H^{\gamma}$, we obtain:
$
\|\boldsymbol{q} - \boldsymbol{q}_h\| = \mathcal{O} ( H^{f(\gamma)}),
$
where 
$
f(\gamma)$ will be the minimum of the  $\alpha(k+1), l, \gamma + k + \frac{1}{2}, \gamma + \alpha(k+1) -\frac{1}{2}$, 
and $-\frac{\gamma}{2} + \alpha(k+1) - \frac{1}{2}\}.$
The above function is continuous with respect to $\gamma$. It is obvious that $f(\gamma)<0$  if $|\gamma| > 2\alpha(k+1)$. Therefore the 
absolute maximum of $f(\gamma)$ appears in the interval $(-2\alpha(k+1), 2\alpha(k+1))$. Assume that $f(\gamma)$ achieves its maximum at 
$\gamma = \gamma^*$, we can take $\tau = H^{\gamma^*}$ to get optimal convergence rate for $\boldsymbol{q}$. 
This strategy can be applied to $u$ as well.

\section{Proof of the main results}\label{proofs}
Now we prove the main results  of the paper stated in Theorems \ref{estimate_q} and  \ref{estimate_u}.
The proofs follow the technique developed in \cite{CockburnQiuShi2012} for the hybridizable 
discontinuous Galerkin method and is done in several steps, by establishing first an estimate 
for the vector variable $ \boldsymbol{q}$ and then for scalar variable $u$.

\subsection{Error equations} \label{error equations}

We begin by obtaining the error equations 
we shall use in the analysis. The main idea is to work with the following projection errors:
\begin{alignat*}{1}
\boldsymbol{e}_q :=&\; \Piv \boldsymbol{q} - \boldsymbol{q}_h,
\\
 e_u :=&\; \Piw u - u_h,
\\
\ehatq \cdot \n :=&\; P_{M} (\boldsymbol{q} \cdot \n) - \qhat \cdot \n,
\\
\ehatu :=&\; P_M u - \uhat. 
\intertext{Further, we define}
\delta_u &:= u - \Piw u, \\
\boldsymbol{\delta}_{q} & := \boldsymbol{q} - \Piv \boldsymbol{q}.
\end{alignat*}
\begin{lemma}\label{error_equation}
Under the Assumption \ref{assum-inclusions}, we have
\begin{subequations}
 \begin{alignat}{3}
  \label{error-equation-1}
 \bint{\alpha \boldsymbol{e}_q}{\boldsymbol{v}} -\bint{e_u}{\nabla \cdot \boldsymbol{v}} &&&+ \bintEh{\ehatu}{\boldsymbol{v} \cdot \n}
&&= -\bint{\alpha \boldsymbol{\delta}_q}{\boldsymbol{v}}
 - \langle {(I - P_M)u}\; , \; \boldsymbol{v} \cdot \n\rangle_{\partial \mathcal{T}_h},\\
\label{error-equation-2}
-\bint{\boldsymbol{e}_q}{\nabla w} & &&+ \bintEh{\ehatq \cdot \n}{w} && = - \bintEh{(I - P_M)(\boldsymbol{q} \cdot \n)}{w},\\
\label{error-equation-3}
& && \quad \langle {\ehatq \cdot \n},{\mu}\rangle_{\partial\Oh} && = 0,\\
\label{error-equation-4}
& && \hspace{1.7cm} {\ehatu}|_{\partial \Omega} &&= 0,
 \end{alignat}
\end{subequations}
for all $(w,\boldsymbol{v},\mu) \in  W_h \times \Vh \times M_{h,H}$. Here $I$ is the identity operator.
Moreover, 
\begin{equation}
\label{ehatq}
\ehatq\cdot\boldsymbol{n}=\boldsymbol{e}_q\cdot\boldsymbol{n}
+ \alphaa(e_u-\ehatu) - (P_h^{\partial} - P_M)(\boldsymbol{q} \cdot \n + \alphaa u)
\quad \mbox{ on } \quad \partial\Oh.
\end{equation}
\end{lemma}

\begin{proof}
 Let us begin by noting that the exact solution $(u,\boldsymbol{q})$ obviously satisfies 
\begin{align*}
 \bint{\alpha \boldsymbol{q}}{\boldsymbol{v}}
-\bint{u}{\nabla \cdot \boldsymbol{v}}  + \bintEh{u}{\boldsymbol{v}\cdot \n}  &= 0,\\
-\bint{\boldsymbol{q}}{\nabla w} + \bintEh{\boldsymbol{q} \cdot \n}{w} &= \bint{f}{w},\\
\langle {\boldsymbol{q} \cdot \n},{\mu}\rangle_{\partial\Oh} &= 0,
\end{align*}
for all $(w,\boldsymbol{v}, \mu) \in  W_h \times \Vh \times M_{h,H}$. 
By the orthogonality properties \eqref{HDG projection-1} and \eqref{HDG projection-2} of the projection $\Pi=(\Piv, \Piw)$,
we obtain that
\begin{align*}
 \bint{\alpha \boldsymbol{q}}{\boldsymbol{v}}
-\bint{\Piw u}{\nabla \cdot \boldsymbol{v}}  + \bintEh{u}{\boldsymbol{v}\cdot \n}  &= 0,\\
-\bint{\Piv \boldsymbol{q}}{\nabla w} + \bintEh{\boldsymbol{q} \cdot \n}{w} &= \bint{f}{w},\\
\langle {\boldsymbol{q} \cdot \n},{\mu}\rangle_{\partial\Oh} &= 0,
\end{align*}
for all $(w,\boldsymbol{v}, \mu) \in W_h \times  \Vh \times  M_{h,H}$. 
Moreover, since $P_M$ is the 
$L^2$-projection into $M_{h,H}$, we get, 
\begin{align*}
 \bint{\alpha \boldsymbol{q}}{\boldsymbol{v}}
-\bint{\Piw u}{\nabla \cdot \boldsymbol{v}}  + \bintEh{P_M u}{\boldsymbol{v}\cdot \n}     &= - \bintEh{(I - P_M) u}{\boldsymbol{v} \cdot \n},\\
-\bint{\Piv \boldsymbol{q}}{\nabla w} + \bintEh{P_M(\boldsymbol{q} \cdot \n)}{w} &= \bint{f}{w}
 - \bintEh{(I - P_M)(\boldsymbol{q} \cdot \n)}{w},\\
\langle {P_M(\boldsymbol{q} \cdot \n)},{\mu}\rangle_{\partial\Oh} &= 0,
\end{align*}
for all $(w, \boldsymbol{v}, \mu) \in  W_h \times \Vh \times M_{h,H}$.
Subtracting the four equations defining the weak formulation of the HDG method
\eqref{weak formulation} from the above equations, respectively, we obtain the equations
for the projection of the errors. The last error equation \eqref{error-equation-4} 
is due to the definition of $\uhat$ on $\partial \Omega$.

It remains to prove the identity \eqref{ehatq}  for $\boldsymbol{e}_{\widehat{q}} \cdot \n$. 
On each $F \in \partial K, \, K \in \mathcal{T}_h$ after using the the definition of numerical traces \eqref{trace-q}
we get
\begin{alignat*}{1}
\boldsymbol{e}_{\widehat{q}}\cdot\n - \boldsymbol{e}_q \cdot \n
       &= P_M(\boldsymbol{q} \cdot \n) - \qhat \cdot \n - (\Piv \boldsymbol{q} \cdot \n - \boldsymbol{q}_h \cdot \n) \\
       & =  P_M(\boldsymbol{q} \cdot \n) - \Piv \boldsymbol{q} \cdot \n- (\qhat \cdot \n  - \boldsymbol{q}_h \cdot \n) \\
       & =  P^{\partial}_h(\boldsymbol{q} \cdot \n) - \Piv \boldsymbol{q} \cdot \n+ \alphaa(u_h - \uhat) 
                        + (P_M - P^h_{\partial})(\boldsymbol{q} \cdot \n). \\
\intertext{Then using the property of the projection $\Piw$  defined in \eqref{bdry_identity} the equality reduces to} 
\boldsymbol{e}_{\widehat{q}}\cdot\n - \boldsymbol{e}_q \cdot \n & =  
      \alphaa(- P^{\partial}_h u + \Piw u)+ \alphaa(u_h - \uhat) + (P_M - P^h_{\partial})(\boldsymbol{q} \cdot \n) \\
& =  \alphaa(- P_M u + \Piw u)+ \alphaa(u_h - \uhat) + (P_M - P^h_{\partial})(\boldsymbol{q} \cdot \n + \alphaa u) \\
& =  \alphaa(e_u - \ehatu) + (P_M - P^h_{\partial})(\boldsymbol{q} \cdot \n + \alphaa u) 
\end{alignat*}
and this completes the proof.
\hfill \end{proof}

\subsection{ Estimate for $\boldsymbol{q} - \boldsymbol{q_h}$}\label{estimate for q}
For the error estimate of $\boldsymbol{e}_q$ we need to following lemma:
\begin{lemma}\label{H_div_conforming}
Let the Assumptions \ref{single-faced HDG}, \ref{tau_positive_EH} hold. Then
\begin{enumerate}
\item[(a)]  on each subdomain $T \in \mathcal{P}$, $\boldsymbol{e}_q \in   \boldsymbol{H}(div, T)$; 
\item[(b)] $\|\nabla \cdot \boldsymbol{e}_q\|_{T} = 0, \, \, $ for all $T \in \mathcal{P}$;
\item[(c)] $ \boldsymbol{e}_q \cdot \n|_F = \ehatq \cdot \n|_F,  \, \, $  for all $F \in \mathcal{E}_h^0(T)$.
\end{enumerate}
\end{lemma}
\begin{proof}
Now take any $T \in \mathcal{P}$.
To prove that $\boldsymbol{e}_q$ is $ \boldsymbol{H}_{div}$-conforming in $T$, we need to show that 
$\boldsymbol{e}_q \cdot \n$ is continuous across all interior interfaces $F \in \mathcal{E}_h^0(T)$. 
By the error equation \eqref{error-equation-3}, we know that $\ehatq \cdot \n$ is single valued on all interior 
interfaces due to the fact that $\ehatq \cdot \n$ and the test function $\mu$ are in the same space $M_h(F)$. 
Hence, it is suffices to show that 
\[
\boldsymbol{e}_q \cdot \n|_{F} = \ehatq \cdot \n|_{F}, \quad \forall \; {F \in  \mathcal{E}_h^0(T)}.
\] 
First of all, on each interior face $P^h_{\partial} = P_M$, together with \eqref{ehatq}, we have
\begin{equation}\label{local_relation}
\ehatq \cdot \n = \boldsymbol{e}_q \cdot \n + \tau(e_u - \ehatu), \quad \forall \; F \in  \mathcal{E}^0_h(T).
\end{equation}
From here we can see that $\boldsymbol{e}_q \cdot \n |_F = \ehatq \cdot \n |_F$ if $\tau|_F = 0$. We only need to show that 
\begin{equation}\label{local_identity}
\tau(e_u -\ehatu)|_{F^*} = 0, \quad \forall \; F^* \in \partial K,~~ F^* \cap \mathcal{E}_H = \emptyset.
\end{equation}
On any $K$ adjacent with $\mathcal{E}_H$, by our assumptions, $\tau > 0$ on $F^*$ where $F^*$ is on the boundary of $T$. 
So on the other faces $\tau = 0$ and hence $\ehatq \cdot \n = \boldsymbol{e}_q \cdot \n$. 

Let us consider an arbitrary interior element $K$ with $\tau > 0$ on $F^*$. We restrict the error equation \eqref{error-equation-2} 
on $K$, integrating by parts, we have
\[
\bintK{\nabla \cdot \boldsymbol{e}_q}{w} + \bintEK{\ehatq \cdot \n - \boldsymbol{e}_q \cdot \n}{w} = - \bintEK{(I - P_M)(\boldsymbol{q} \cdot \n)}{w}.
\]
By \eqref{local_relation} and the fact that $P_M = P^h_{\partial}$ on $\partial K$, we have
\[
\bintK{\nabla \cdot \boldsymbol{e}_q}{w} + \bintEK{\tau(e_u - \ehatu)}{w} = 0.
\]
Since $\tau > 0$ only on $F^*$, we have
\[
\bintK{\nabla \cdot \boldsymbol{e}_q}{w} + \langle \; \tau(e_u - \ehatu)\;,\; w\; \rangle_{F^*} = 0.
\]
Now let  $w \in P^k(K)$ be such that 
\begin{subequations}\label{local_lifting}
\begin{align}
\bintK{w}{r} &= \bintK{\nabla \cdot \boldsymbol{e}_q}{r}, \quad & \forall \; r \in P^{k-1}(K), \\
\langle  w \; , \; \mu \rangle_{F^*} &= \langle  e_u - \ehatu \; , \; \mu \rangle_{F^*} 
& \forall \; \mu \in P^k(F^*).
\end{align}
\end{subequations}
One can easily see that such $w  \in P^k(K)$ exists and is unique.  Indeed,  
this is  a square system for the coefficients of the polynomial $w$ and it is sufficient to show that the homogeneous 
system has only a trivial solution. On $F^*$ the equation $\langle  w  ,  \mu \rangle_{F^*}=0$ represents
a square homogeneous system for the trace $w|_{F^*} \in P^k(F^*)$. This ensures that the trace  is
identically zero on $F^*$. Without loss of generality we can assume that $F^*$ is in the hyperplane $x_1=0$.
Then obviously $w=x_1 \tilde w$ with $\tilde  w \in P^{k-1}(K)$
	and now  $( x_1 \tilde  w  , \; r )_K = 0 $ for all $ r \in P^{k-1}(K)$ implies $\tilde w =0$.
Then we plug $w$ into the above error equation and notice that 
$\nabla \cdot \boldsymbol{e}_{q} \in P^{k-1}(K), ~e_u - \ehatu \in P^k(F^*)$ to  get
\[
\bintK{\nabla \cdot \boldsymbol{e}_q}{ \nabla \cdot \boldsymbol{e}_q} + \langle \; \tau(e_u - \ehatu)\; , \; e_u - \ehatu \; \rangle_{F^*} = 0.
\]
This implies 
\[
\nabla \cdot \boldsymbol{e}_q|_K = 0, \quad e_u - \ehatu|_{F^*} = 0
\]
and hence, $\boldsymbol{e}_q \cdot \n|_{F} = \ehatq \cdot \n|_{F}$ for all $ F \in  \mathcal{E}^0_h(K)$.
Consequently, $\boldsymbol{e}_q \in \boldsymbol{H}(div, T)$ for all $T \in \mathcal{P}$. 

To finish, we still need to show that $\nabla \cdot \boldsymbol{e}_q|_K = 0$ when $K$ is adjacent with the boundary of $T$. 
Similarly as interior element $K$, error equation \eqref{error-equation-2} gives
\[
\bintK{\nabla \cdot \boldsymbol{e}_q}{w} + \langle \; \tau(e_u - \ehatu)\;,\; w\; \rangle_{F^*} = - \langle \; (I - P_M) (\boldsymbol{q} \cdot \n) \; ,\; w \; \rangle_{F^*}.
\]
Take $w$ to be again the unique element in $P^k(K)$ such that
$$
\bintK{w}{r}  = \bintK{\nabla \cdot \boldsymbol{e}_q}{r} \quad  \forall \, r\in P^{k-1}(K) \quad \mbox{and} \quad
\langle  w \;,\; \mu  \rangle_{F^*}  = 0 \quad \forall \, \mu \in P^k(F^*).
$$
The second equation implies that $w = 0$ on $F^*$, so we have
\[\bintK{\nabla \cdot \boldsymbol{e}_q}{w} + \langle  \tau(e_u - \ehatu)\;,\; w \rangle_{F^*}  
= \bintK{\nabla \cdot \boldsymbol{e}_q}{\nabla \cdot \boldsymbol{e}_q} =0 
\quad \rightarrow  \quad \nabla \cdot \boldsymbol{e}_q=0.
\]
This completes the proof.
\hfill \end{proof}
\begin{remark}The above proof cannot be applied for $\mathbf{BDM}_k$. 
Namely, a key step is the special choice of $w$ which satisfies \eqref{local_lifting}. 
In the case of $\mathbf{BDM}_k$, $w$ is in a smaller space $P^{k-1}(K)$, 
hence the existence of $w$ is no longer valid.
\end{remark}

We are now ready to obtain an upper bound of the $L^2$-norm of $\boldsymbol{e}_q$.
We first prove the following Lemma.

\begin{lemma}\label{energy_argument}
Under Assumption \ref{assum-inclusions}, we have 
\begin{alignat*}{1}
\|\boldsymbol{e}_q\|^2_{\alpha, \Omega} + \|e_u- \ehatu\|^2_{\alphaa, \partial \mathcal{T}_h}
& = -\bint{\alpha \boldsymbol{\delta}_q}{\boldsymbol{e}_q} - \bintEH{(I - P^H_{\partial}) u}{\boldsymbol{e}_q \cdot \n} \\
& + \bintEH{P^h_{\partial} u  - P^H_{\partial} u}{\tau(e_u - \ehatu)} - \bintEH{(I - P^h_{\partial})(\boldsymbol{q}\cdot \n)}{e_u - \ehatu},
\end{alignat*}
where 
$$\|\boldsymbol{e}_q\|^2_{\alpha, \Omega}:= \bint{\alpha \boldsymbol{e}_q}{\boldsymbol{e}_q}, \quad \mbox{and}  \quad
\|e_u- \ehatu\|^2_{\alphaa, \mathcal{E}_h} = \bintEh{\alphaa(e_u - \ehatu)}{e_u - \ehatu}.$$
\end{lemma}

\begin{proof}
By the error equation \eqref{error-equation-4} we know that $\ehatu \in M^0_{h,H}$. 
Taking $(\boldsymbol{v}, w, \mu) = (\boldsymbol{e}_q, e_u, \ehatu)$ in the error 
equations \eqref{error-equation-1}-\eqref{error-equation-3}
respectively and adding, we get, after some algebraic manipulation,
\begin{align*}
\|\boldsymbol{e}_q\|^2_{\alpha, \Omega} - \bintEh{\boldsymbol{e}_q \cdot \n- \ehatq \cdot \n}{e_u - \ehatu} = 
 & - \bint{\alpha \boldsymbol{\delta}_q}{\boldsymbol{e}_q} - \bintEh{(I - P_M)u}{\boldsymbol{e}_q \cdot \n}\\
 &- \bintEh{(I-P_M)(\boldsymbol{q} \cdot \n) }{e_u}.
\end{align*}
Inserting the identity \eqref{ehatq} in the above equation, we get
\begin{align*}
\|\boldsymbol{e}_q\|^2_{\alpha, \Omega} + \|e_u- \ehatu\|^2_{\alphaa, \partial \mathcal{T}_h}= 
   & - \bint{\alpha \boldsymbol{\delta}_q}{\boldsymbol{e}_q} - \bintEh{(I - P_M)u}{\boldsymbol{e}_q \cdot \n} \\
& - \bintEh{(I-P_M)(\boldsymbol{q} \cdot \n) }{e_u} \\
& + \bintEh{(P^h_{\partial}-P_M)(\boldsymbol{q} \cdot \n + \alphaa u)}{e_u - \ehatu}\\
= & - \bint{\alpha \boldsymbol{\delta}_q}{\boldsymbol{e}_q} - \bintEh{(I - P_M)u}{\boldsymbol{e}_q \cdot \n} \\
&- \bintEh{(I-P_M)(\boldsymbol{q} \cdot \n) }{e_u -\ehatu} + \bintEh{(P^h_{\partial}-P_M)(\boldsymbol{q} \cdot \n)}{e_u - \ehatu} \\
&+ \bintEh{(P^h_{\partial}-P_M)(\tau u)}{e_u - \ehatu}\\
\intertext{Now using  the fact that $\ehatu$ is single valued on $\mathcal{E}_h$ and $\ehatu = 0$ on $\partial \Omega$ we get}
\|\boldsymbol{e}_q\|^2_{\alpha, \Omega} + \|e_u- \ehatu\|^2_{\alphaa, \partial \mathcal{T}_h}
= & - \bint{\alpha \boldsymbol{\delta}_q}{\boldsymbol{e}_q} - \bintEh{(I - P_M)u}{\boldsymbol{e}_q \cdot \n}\\
& + \bintEh{P^h_{\partial} u-P_M u}{\tau(e_u - \ehatu)} - \bintEh{(I - P^h_{\partial})(\boldsymbol{q}\cdot \n)}{e_u - \ehatu}.
\end{align*}
Finally noticing  that on each $F \in \partial \mathcal{T}_h, F \cap \mathcal{E}_H = \emptyset$
\[
P_M = P^h_{\partial}, \quad e_u|_F, \ehatu|_F, \boldsymbol{e}_q \cdot \n|_F \in M_h(F),
\]
we get the identity 
\begin{align*}
- \bintEh{(I & - P_M)u}{\boldsymbol{e}_q \cdot \n} 
     + \bintEh{P^h_{\partial} u-P_M u}{\tau(e_u - \ehatu)} - \bintEh{(I - P^h_{\partial})(\boldsymbol{q}\cdot \n)}{e_u - \ehatu} \\
= & - \bintEH{(I - P^H_{\partial})u}{\boldsymbol{e}_q \cdot \n} 
       + \bintEH{P^h_{\partial} u-P^H_{\partial} u}{\tau(e_u - \ehatu)} - \bintEH{(I - P^h_{\partial})(\boldsymbol{q}\cdot \n)}{e_u - \ehatu}, 
\end{align*}
which completes the proof.
\hfill \end{proof}

Now we are ready to present our first estimate for $\boldsymbol{e}_q$:

\begin{theorem}\label{error_q}
If Assumption \ref{assum-inclusions} -\ref{tau_positive_EH} hold, then we have
\begin{align*}
\|\boldsymbol{e}_q\|_{\alpha, \Omega} + \|e_u - \ehatu\|_{\tau, \partial \mathcal{T}_h} 
\le C_{\alpha} \|\boldsymbol{\delta}_q\|_{\alpha,\Omega}  & 
+ C (H^{t - \frac{1}{2}} L^{-\frac{1}{2}}  + \tau H^t L^{-\frac{1}{2}}) \|u\|_{t+ 1, \Om}\\
& 
+ C\tau h^s L^{-\frac{1}{2}} \|u\|_{s+1, \Omega} + C \tau^{-\frac{1}{2}} h^s L^{-\frac{1}{2}}\|\boldsymbol{q}\|_{s+1, \Omega},
\end{align*}
for all $0 \le s \le k+1, \; 0 \le t \le l + 1$. The constants $C$ are independent of the mesh size $h, H$. $C_{\alpha}$ solely depends on $\alpha$. 
\end{theorem}

\begin{proof}
We recall that by definition  
$\|\mu\|^2_{t,\partial \mathcal{T}_H}:= \sum_{T \in \mathcal{P}} \|\mu\|^2_{t, \partial T}$.
We begin by giving an alternative expression for $\bintEH{u - P^H_{\partial} u}{\boldsymbol{e}_q \cdot \n}$.  
Using that $\mathcal{I}^0_H u - P^H_{\partial} u \in M_H$ and  the equation \eqref{error-equation-3} we get
\begin{align*}
\bintEH{u - P^H_{\partial} u}{\boldsymbol{e}_q \cdot \n} & 
         = \bintEH{u - \mathcal{I}^0_H u}{\boldsymbol{e}_q \cdot \n} + \bintEH{\mathcal{I}^0_H u - P^H_{\partial} u}{\boldsymbol{e}_q \cdot \n} \\
& = \bintEH{u - \mathcal{I}^0_H u}{\boldsymbol{e}_q \cdot \n} + \bintEH{\mathcal{I}^0_H u - P^H_{\partial} u}{\boldsymbol{e}_q \cdot \n - \ehatq \cdot \n}.
\end{align*}
Then further using  \eqref{ehatq} we get
\begin{align*}
\bintEH{u - P^H_{\partial} u}{\boldsymbol{e}_q \cdot \n}
& = \bintEH{u - \mathcal{I}^0_H u}{\boldsymbol{e}_q \cdot \n} - \bintEH{\mathcal{I}^0_H u - P^H_{\partial} u}{\tau(e_u - \ehatu)}\\
& \quad + \bintEH{\mathcal{I}^0_H u - P^H_{\partial} u }{(P^h_{\partial} - P_M)(\boldsymbol{q} \cdot \n + \tau u)}.
\end{align*}
Then  using Lemma \ref{H_div_conforming} we get  the estimate:
\begin{align*}
\bintEH{u - \mathcal{I}^0_H u}{\boldsymbol{e}_q \cdot \n} 
& = \sum_{T \in \mathcal{P}} \|u - \mathcal{I}_H^0 u \|_{\frac{1}{2}, \partial T} \|\boldsymbol{e}_q\|_{H(div ,T)}\\
 &=  \sum_{T \in \mathcal{P}} \|u - \mathcal{I}_H^0 u \|_{\frac{1}{2}, \partial T} \|\boldsymbol{e}_q\|_{T)}\\
 &\le C H^{t-\frac{1}{2}} L^{-\frac{1}{2}} \|u\|_{t+1, \Om} \|\boldsymbol{e}_q\|_{\Omega},
\end{align*}
for all $0 \le t \le l+1$, where in the  last step we used Lemma \ref{bdy_to_interior}.

By the previous identity and Lemma \ref{energy_argument}, we have  
\begin{align*}
\|\boldsymbol{e}_q\|^2_{\alpha, \Omega} & + \|e_u- \ehatu\|^2_{\alphaa, \partial \mathcal{T}_h}
 = -\bint{\alpha \boldsymbol{\delta}_q}{\boldsymbol{e}_q} -\bintEH{u - \mathcal{I}^0_H u}{\boldsymbol{e}_q \cdot \n} \\
&+ \bintEH{\mathcal{I}^0_H u - P^H_{\partial} u}{\tau(e_u - \ehatu)}
 - \bintEH{\mathcal{I}^0_H u - P^H_{\partial} u }{(P^h_{\partial} - P^H_{\partial})(\boldsymbol{q} \cdot \n + \tau u)} \\
& + \bintEH{P^h_{\partial} u  - P^H_{\partial} u}{\tau(e_u - \ehatu)} - \bintEH{(I - P^h_{\partial})(\boldsymbol{q}\cdot \n)}{e_u - \ehatu}\\
&\le \|\boldsymbol{e}_q\|_{\alpha, \Omega} \|\boldsymbol{\delta}_{q}\|_{\alpha, \Omega} 
+ C H^{t-\frac{1}{2}} L^{-\frac{1}{2}} \|u\|_{t+1, \Om} \|\boldsymbol{e}_q\|_{\Omega} 
 +\tau^\frac12 \|\mathcal{I}^0_H u - P^H_{\partial} u\|_{\partial \mathcal{T}_H} \|e_u - \ehatu\|_{\tau, \partial \mathcal{T}_H} \\
&+ \|\mathcal{I}^0_H u - P^H_{\partial} u\|_{\partial \mathcal{T}_H} \left (\|(I - P^h_{\partial}) (\boldsymbol{q} \cdot \n)\|_{\partial \mathcal{T}_H} 
            + \tau \|u - P^h_{\partial} u\|_{\partial \mathcal{T}_H} \right ) \\
& + \tau^{\frac{1}{2}}\|P^h_{\partial} u - P^H_{\partial} u \|_{\partial \mathcal{T}_H} \|e_u - \ehatu\|_{\tau,\partial \mathcal{T}_H} 
+ \tau^{-\frac{1}{2}}\|(I - P^h_{\partial})(\boldsymbol{q} \cdot \n)\|_{\partial \mathcal{T}_H} \|e_u - \ehatu\|_{\tau, \partial \mathcal{T}_H}.
\end{align*}
By using Young's inequality and Lemma \ref{bdy_to_interior}, after some algebraic manipulations, we obtain
\begin{align*}
\|\boldsymbol{e}_q\|_{\alpha, \Omega} + \|e_u - \ehatu\|_{\tau, \partial \mathcal{T}_h} & 
\le C_{\alpha} \|\boldsymbol{\delta}_q\| 
+C (H^{t - \frac{1}{2}} L^{-\frac{1}{2}}+  \tau H^t L^{-\frac{1}{2}} )\|u\|_{t+ 1, \Om}\\
& 
   + C\tau h^s L^{-\frac{1}{2}} \|u\|_{s+1, \Om} 
   + C \tau^{-\frac{1}{2}} h^s L^{-\frac{1}{2}}\|\boldsymbol{q}\|_{s+1, \Om},
\end{align*}
for all $0 \le t \le l+1, \; 0 \le s \le k+1$. This completes the proof.
\hfill \end{proof}

As a consequence, by triangle inequality, we immediately have the estimate for $\boldsymbol{q} - \boldsymbol{q}_h$:
\begin{align*}
\|\boldsymbol{q} - \boldsymbol{q}_h\|_{\alpha, \Omega}  &
\le (C_{\alpha} + 1)\|\boldsymbol{q} -\boldsymbol{\Pi}_V \boldsymbol{q}\|
 + C H^{t - \frac{1}{2}} L^{-\frac{1}{2}} \|u\|_{t+1, \Om}\\
& +C \tau H^t L^{-\frac{1}{2}} \|u\|_{t+1, \Om} + C\tau h^s L^{-\frac{1}{2}} \|u\|_{s+1, \Om} + C \tau^{-\frac{1}{2}} h^s L^{-\frac{1}{2}}\|\boldsymbol{q}\|_{s+1, \Om},
\end{align*}
for $1 \le s \le k+1, 1 \le t \le l+1$.

From this estimate we see that we may have various scenarios in choosing the scales $L$ and $H$ and the stabilization 
parameter $\alphaa$. Some of these were discussed in Subsection \ref{ssec:main}
For example, we take $\tau = \mathcal{O}(1)$ and assume $L = \mathcal{O}(1)$, 
then $\|\boldsymbol{e}_q\|_0, \|e_u - \ehatu\|_{\tau, \partial \mathcal{T}_h}$ has the 
order as $\mathcal{O}(h^{k+1} + H^{l+ \frac{1}{2}})$, which is the same as the result in \cite{Arbogast_PWY_07}.

\begin{remark} It important to note that the fact that $\boldsymbol{e}_q \in H_{div}$ is essential
in obtaining an optimal order of convergence. If $\boldsymbol{e}_q$ were not 
$H_{div}$-conforming, then we will have convergence rate 
$\mathcal{O}(h^{-\frac{1}{2}}H^{l+1})$. In the proof, $H_{div}$-conformity of 
the vector field $ \boldsymbol{e}_q$  depends essentially on the fact that $\tau$ is single faced. 
It will be interesting to see what kind of numerical result we have if this assumption is failed.
\end{remark}

\subsection{Estimate for $u - u_h$} \label{estimate for u}

Using a standard elliptic duality argument, we have the following result:
\begin{lemma}\label{duality}
We have
\begin{align}\label{s1-4}
\|e_u\|^2_{\Omega}= \mathbb{S}_1 + \mathbb{S}_2 + \mathbb{S}_3 + \mathbb{S}_4,
\end{align}
where 
\begin{align*}
\mathbb{S}_1 &= - \bint{\alpha \boldsymbol{e}_q}{\boldsymbol{\theta} - 
     \Piv \boldsymbol{\theta}} + \bint{\alpha \boldsymbol{\delta}_q}{ \Piv \boldsymbol{\theta}}, \\
\mathbb{S}_2 & = \bintEh{e_u - \ehatu}{\boldsymbol{\theta} \cdot \n - P^{\partial}_h(\boldsymbol{\theta} \cdot \n)} 
      - \bintEh{\boldsymbol{e}_q \cdot \n - \ehatq \cdot \n}{\phi - P^{\partial}_h \phi} ,\\
\mathbb{S}_3 & = -\bintEh{(P^{\partial}_h - P_M)(\boldsymbol{q} \cdot \n)}{P^{\partial}_h \phi} 
- \bintEh{(P^{\partial}_h - P_M) u }{P^{\partial}_h (\boldsymbol{\theta} \cdot \n}, \\
\mathbb{S}_4 &= - \bintEH{\boldsymbol{e}_q \cdot \n}{\phi - \mathcal{I}_H^0 \phi} 
      + \bintEH{\boldsymbol{e}_q \cdot \n - \ehatq \cdot \n}{\phi - \mathcal{I}_H^0 \phi}.
\end{align*} 
\end{lemma}

\begin{proof}
We begin by using the second equation \eqref{adjoint-2} of the dual problem to write that
\begin{alignat*}{1}
  \bint{e_u}{e_u}  =&\; \bint{e_u}{\nabla \cdot \boldsymbol{\theta}} \\
      =&\; \bint{e_u}{\nabla \cdot \boldsymbol{\theta}} - 
             \bint{\boldsymbol{e}_q}{\alpha \boldsymbol{\theta}} - \bint{\boldsymbol{e}_q}{\nabla \phi},
\end{alignat*}
{by the first equation \eqref{adjoint-1} of the dual problem. This implies that}
\begin{alignat*}{1}
   \bint{e_u}{e_u} =&\; \bint{e_u}{\nabla \cdot \Piv \boldsymbol{\theta}} 
         - \bint{\alpha \boldsymbol{e}_q}{\Piv \boldsymbol{\theta}} - \bint{\boldsymbol{e}_q}{\nabla \Piw \phi}\\
                   & + \bint{e_u}{\nabla \cdot (\boldsymbol{\theta} - \Piv \boldsymbol{\theta})} 
        - \bint{\alpha \boldsymbol{e}_q}{\boldsymbol{\theta} - \Piv \boldsymbol{\theta}} - \bint{\boldsymbol{e}_q}{\nabla (\phi - \Piw \phi)}.
 \end{alignat*} 
Taking $\boldsymbol{v} := \Piv \boldsymbol{\theta}$ {in the first }error equation, \eqref{error-equation-1},
and $w := \Piw \phi$ in the second, \eqref{error-equation-2}, we obtain
\begin{alignat*}{1}
  \bint{e_u}{e_u} =&\; \bint{\alpha \boldsymbol{\delta}_q }{ \Piv \boldsymbol{\theta}}
  + \bintEh{\ehatu}{ \Piv\boldsymbol{\theta} \cdot \n}
  - \bintEh{\ehatq \cdot \n}{\Piw \phi}\\
                   & + \bint{e_u}{\nabla \cdot (\boldsymbol{\theta} - \Piv \boldsymbol{\theta})} 
                     - \bint{\alpha \boldsymbol{e}_q}{\boldsymbol{\theta} - \Piv \boldsymbol{\theta}} - \bint{\boldsymbol{e}_q}{\nabla (\phi - \Piw \phi)}\\
& + \bintEh{(I - P_M)u}{\Piv \boldsymbol{\theta} \cdot \n} - \bintEh{(I - P_M)(\boldsymbol{q} \cdot \n)}{\Piw \phi}
\end{alignat*}
and, after simple algebraic manipulations we get
\begin{equation}\label{eu-inner-prod}  \bint{e_u}{e_u}
=- \bint{\alpha \boldsymbol{e}_q}{\boldsymbol{\theta} - \Piv \boldsymbol{\theta}} + \bint{\alpha \boldsymbol{\delta}_q}{ \Piv \boldsymbol{\theta}} +\mathbb{T},
\end{equation}
where
\begin{align*}
 \mathbb{T}:= & \bintEh{\ehatu}{ \Piv\boldsymbol{\theta} \cdot \n} - \bintEh{\ehatq \cdot \n}{\Piw \phi}\\
& + \bintEh{(I - P_M)u}{\Piv \boldsymbol{\theta} \cdot \n} - \bintEh{(I - P_M)(\boldsymbol{q} \cdot \n)}{\Piw \phi}\\
                   & + \bint{e_u}{\nabla \cdot (\boldsymbol{\theta} - \Piv \boldsymbol{\theta})} - \bint{\boldsymbol{e}_q}{\nabla (\phi - \Piw \phi)}.\\
\intertext{Integrating by parts for the last two terms and applying the projection properties \eqref{HDG projection-1}, \eqref{HDG projection-2} we have,}
\mathbb{T}=& \bintEh{\ehatu}{ \Piv\boldsymbol{\theta} \cdot \n} - \bintEh{\ehatq \cdot \n}{\Piw \phi}\\
& + \bintEh{(I - P_M)u}{\Piv \boldsymbol{\theta} \cdot \n} - \bintEh{(I - P_M)(\boldsymbol{q} \cdot \n)}{\Piw \phi}\\
                   & + \bintEh{e_u}{(\boldsymbol{\theta} - \Piv \boldsymbol{\theta}) \cdot \n} 
                    - \bintEh{\boldsymbol{e}_q \cdot \n}{\phi - \Piw \phi} : = \mathbb{T}_1 + \mathbb{T}_2,
\end{align*}
where
\begin{align*}
\mathbb{T}_1 := & \bintEh{e_u - \ehatu}{(\boldsymbol{\theta} - \Piv \boldsymbol{\theta}) \cdot \n} - \bintEh{\boldsymbol{e}_q \cdot \n - \ehatq \cdot \n}{\phi - \Piw \phi} \\
\mathbb{T}_2 := & \bintEh{(Id - P_M)u}{\Piv \boldsymbol{\theta} \cdot \n} - \bintEh{(Id - P_M)(\boldsymbol{q} \cdot \n)}{\Piw \phi}\\
& + \bintEh{\ehatu}{\boldsymbol{\theta} \cdot \n} - \bintEh{\ehatq \cdot \n}{\phi}.
\end{align*}
We will estimate $\mathbb{T}_1, \mathbb{T}_2$ separately. First we transform $\mathbb{T}_1$ by adding and subtracting
the terms $  P^{\partial}_h \boldsymbol{\theta} \cdot \n$ and $  P^{\partial}_h \phi$ to get
 \begin{align*}
\mathbb{T}_1
=& \bintEh{e_u - \ehatu}{P^{\partial}_h(\boldsymbol{\theta} \cdot \n) - \Piv \boldsymbol{\theta} \cdot \n} - \bintEh{\boldsymbol{e}_q \cdot \n - \ehatq \cdot \n}{P^{\partial}_h \phi - \Piw \phi} \\
&+ \bintEh{e_u - \ehatu}{\boldsymbol{\theta} \cdot \n - P^{\partial}_h(\boldsymbol{\theta} \cdot \n)} - \bintEh{\boldsymbol{e}_q \cdot \n - \ehatq \cdot \n}{\phi - P^{\partial}_h \phi}. 
\end{align*}
Then using the identity
$\boldsymbol{e}_q \cdot \n - \ehatq \cdot \n = \tau (\Piv u - P^{\partial}_h u)$, a simple consequence of the projection property \eqref{bdry_identity},  and error equation \eqref{ehatq}
we get
\begin{align*}
\mathbb{T}_1
&= \bintEh{e_u - \ehatu}{\boldsymbol{\theta} \cdot \n - P^{\partial}_h(\boldsymbol{\theta} \cdot \n)} - \bintEh{\boldsymbol{e}_q \cdot \n - \ehatq \cdot \n}{\phi - P^{\partial}_h \phi} \\
&-\bintEh{(P^{\partial}_h - P_M)(\boldsymbol{q} \cdot \n)}{P^{\partial}_h \phi - \Piw \phi} 
- \bintEh{(P^{\partial}_h - P_M) u }{\tau(P^{\partial}_h \phi - \Piw \phi)}.
\end{align*}
Next, we transform the expression $\mathbb{T}_2  $ by taking into account that $\bintEh{\ehatu}{\boldsymbol{\theta} \cdot \n} = 0$ and using  
the fact that $\Piw u|_F, \Piv \boldsymbol{q} \cdot \n|_F \in M(F)$ for any $F \in \partial \mathcal{T}_h$:
\begin{align*}
\mathbb{T}_2 & =  \bintEh{(I - P_M)u}{\Piv \boldsymbol{\theta} \cdot \n} 
                            - \bintEh{(I - P_M)(\boldsymbol{q} \cdot \n)}{\Piw \phi} - \bintEh{\ehatq \cdot \n}{\phi} \\
  & = \bintEh{(P^{\partial}_h - P_M)u}{\Piv \boldsymbol{\theta} \cdot \n} 
                            - \bintEh{(P^{\partial}_M - P_M)(\boldsymbol{q} \cdot \n)}{\Piw \phi} - \bintEh{\ehatq \cdot \n}{\phi}. 
\end{align*}
%
Now combining $\mathbb{T}_1, \mathbb{T}_2$  and using the property  \eqref{bdry_identity} of the projections $\Piv$ and $\Piw $, we get
\begin{align*}
\mathbb{T}=
& \bintEh{e_u - \ehatu}{\boldsymbol{\theta} \cdot \n - P^{\partial}_h(\boldsymbol{\theta} \cdot \n)} - \bintEh{\boldsymbol{e}_q \cdot \n - \ehatq \cdot \n}{\phi - P^{\partial}_h \phi} \\
&-\bintEh{(P^{\partial}_h - P_M)(\boldsymbol{q} \cdot \n)}{P^{\partial}_h \phi} 
- \bintEh{(P^{\partial}_h - P_M) u }{P^{\partial}_h (\boldsymbol{\theta} \cdot \n}  - \bintEh{\ehatq \cdot \n}{\phi}.
\end{align*} 

To obtain the final estimate, on each interior face $F \in \mathcal{E}^0_h$, $M_{h,H}|_F = M_h(F)$ and $\ehatq \cdot \n \in M_h(F)$, by \eqref{error-equation-3}, we 
note that  $\ehatq \cdot \n$ is single valued on $F \in \mathcal{E}^0_h$. Then we rewrite the last term as follows:
\begin{align*}
\bintEh{\ehatq \cdot \n}{\phi} & = \bintEH{\ehatq \cdot \n}{\phi}= \bintEH{\ehatq \cdot \n}{\phi - \mathcal{I}_H^0 \phi} \\
& = \bintEH{\boldsymbol{e}_q \cdot \n}{\phi - \mathcal{I}_H^0 \phi} - \bintEH{\boldsymbol{e}_q \cdot \n - \ehatq \cdot \n}{\phi - \mathcal{I}_H^0 \phi},
\end{align*}
where  $ \mathcal{I}_H^0$ is a $ C^0$ Lagrange interpolant defined in \eqref{projections}. At
 the final step we have used the error equation \eqref{error-equation-3} and 
the fact that $\phi|_{\partial \Omega} = \mathcal{I}_H^0 \phi|_{\partial \phi} = 0$

Inserting the above expression into $\mathbb{T}_1 + \mathbb{T}_2$, we finally obtain:
\begin{align*}
\|e_u\|^2_{\Omega}=& - \bint{\alpha \boldsymbol{e}_q}{\boldsymbol{\theta} 
      - \Piv \boldsymbol{\theta}} + \bint{\alpha \boldsymbol{\delta}_q}{ \Piv \boldsymbol{\theta}} \\
&+ \bintEh{e_u - \ehatu}{\boldsymbol{\theta} \cdot \n - P^{\partial}_h(\boldsymbol{\theta} \cdot \n)} 
      - \bintEh{\boldsymbol{e}_q \cdot \n - \ehatq \cdot \n}{\phi - P^{\partial}_h \phi} \\
&-\bintEh{(P^{\partial}_h - P_M)(\boldsymbol{q} \cdot \n)}{P^{\partial}_h \phi} 
- \bintEh{(P^{\partial}_h - P_M) u }{P^{\partial}_h (\boldsymbol{\theta} \cdot \n} \\
& - \bintEH{\boldsymbol{e}_q \cdot \n}{\phi - \mathcal{I}_H^0 \phi} 
      + \bintEH{\boldsymbol{e}_q \cdot \n - \ehatq \cdot \n}{\phi - \mathcal{I}_H^0 \phi},
\end{align*} 
which completes the proof. \hfill
\end{proof}

Notice that in Lemma \ref{duality} we established bounds for the projection errors:
$
\boldsymbol{\theta} \cdot \n - P^{\partial}_h (\boldsymbol{\theta} \cdot \n), \, \phi - P^{\partial}_h \phi$, and $  \phi - \mathcal{I}^0_H \phi.
$
However, here we cannot apply the trace estimates of Lemma \ref{bdy_to_interior} for these terms since 
the solution of the dual problem $(\phi, \boldsymbol{\theta})$ is only in $H^2(\Omega) \times \boldsymbol{H}^1(\Omega)$. 
Alternatively, we will bound these terms based on the following result:
\begin{lemma}\label{dual_trace_inequality}
If the function $(\phi, \boldsymbol{\theta}) \in H^2(T) \times \boldsymbol{H}^1(T)$, then we have
\begin{subequations}\label{dual_trace}
\begin{alignat}{1}
\label{dual_trace_1}
\|\phi - P^{\partial}_h \phi\|_{\partial T} &\le C h^{\frac{3}{2}} \|\phi\|_{2,T}, \\
\label{dual_trace_2}
\|\boldsymbol{\theta}\cdot \n - P^{\partial}_h (\boldsymbol{\theta} \cdot \n)\|_{\partial T} 
    &\le C h^{\frac{1}{2}} \|\boldsymbol{\theta}\|_{1, T},\\
\label{dual_trace_3}
\|\phi - \mathcal{I}^0_H \phi\|_{\partial T} & \le C H^{\frac{3}{2}} \|\phi\|_{2,T}, \\
\label{dual_trace_4}
\|\phi - \mathcal{I}^0_H \phi\|_{\frac{1}{2}, \partial T} & \le C H \|\phi\|_{2,T}.
\end{alignat}
\end{subequations}
\end{lemma}
\begin{proof} 
For \eqref{dual_trace_1}, on each $K \in \mathcal{T}_h$, let $\Pi_h \phi$ denotes the 
$L^2-$projection of $\phi$ onto the local space $W(K)$. We have
\begin{align*}
\|\phi - P^{\partial}_h \phi\|_{\partial T} &\le \sum_{K \in \mathcal{T}_h(T)} \|\phi 
   - P^{\partial}_h \phi\|_{\partial K} \le \sum_{K \in \mathcal{T}_h(T)} \|\phi - \Pi_h \phi\|_{\partial K} \\
& \le C \sum_{K \in \mathcal{T}_h(T)}(h^{-\frac{1}{2}} \|\phi - \Pi_h \phi \|_K 
     + h^{\frac{1}{2}} \|\nabla(\phi - \Pi_h \phi)\|_K) \quad \text{by the trace inequality \eqref{eq:trace},} \\
& \le C \sum_{K \in \mathcal{T}_h(T)} h^{\frac{3}{2}}\|\phi\|_{2,K} \le C h^{\frac{3}{2}} \|\phi\|_{2,T}.
\end{align*}
The trace inequality \eqref{dual_trace_2} can be proven in a similar way. 

To prove \eqref{dual_trace_3}, we proceed as follows:
on each subdomain $T$, based on the partition $\mathcal{E}_H(T)$, we can generate a conforming shape-regular 
triangulation $\mathcal{T}_H(T)$. Therefore, we can also extend the boundary interpolation $\mathcal{I}^0_H \phi$ 
to the whole domain $T$, denoted by $\widetilde{\mathcal{I}^0_H} \phi$. We have
\[
\|\phi - \mathcal{I}^0_H \phi\|_{\partial T} \le \sum_{K_H \in \mathcal{T}_H(T)} \|\phi - \widetilde{\mathcal{I}^0_H} \phi\|_{\partial T}\|_{\partial K_H} 
       \le C H^{\frac{3}{2}} \|\phi\|_{2,T}.
\]
In the last step we've applied the same trace inequality \eqref{eq:trace} and the standard interpolation approximation property. 

For the last inequality, by the definition of $\|\cdot\|_{\frac{1}{2}, \partial T}$, we have:
\begin{align*}
\|\phi - \mathcal{I}^0_H \phi\|_{\frac{1}{2}, \partial T} & = \min_{\tilde{w} \in H^1(T), \tilde{w}|_{\partial T} 
       = \phi - \mathcal{I}^c_H \phi } \|w\|_{1,T}  \le \|\phi - \widetilde{\mathcal{I}^c_H} \phi\|_{1,T}  \le C H \|\phi\|_{2,T}.
\end{align*}
This completes the proof.
\hfill \end{proof}

Now we are ready to establish the final estimate for $e_u$.
\begin{theorem}\label{error_u}
Let the assumptions of Theorem \ref{error_q} hold and, in addition, let the local space $W(K)$ 
contain  piecewise linear functions for any $K \in \mathcal{T}_h$, i.e. $k \ge 1$.  Then 
\begin{align*}
\|e_u\|_{\Omega}  \le & \; (C_{\alpha}h+C H)\|\boldsymbol{e}_q\|_{\alpha, \Omega} + C_{\alpha}h \|\boldsymbol{\delta}_q\|_{\alpha, \Omega}
       + C(h^{\frac{1}{2}} \tau^{-\frac{1}{2}}
              +\tau^{\frac{1}{2}}  H^{\frac{3}{2}} )\|e_u - \ehatu\|_{\tau,\partial \mathcal{T}_h} \\
& + C H^{\frac{3}{2}}L^{-\frac{1}{2}}h^s (\|\boldsymbol{q}\|_{s+ 1} + \tau \|u\|_{s+ 1}) 
+ C H^t L^{-\frac{1}{2}} (H^{\frac{3}{2}}\|\boldsymbol{q}\|_{t+1} + (h^{\frac{1}{2}} + \tau H^{\frac{3}{2}})\|u\|_{t+1} ),
\end{align*}
for all $1 < s \le k+1, 0 \le t \le l+1$.
\end{theorem}
\begin{proof}
We will estimate the terms $\mathbb{S}_j,  \, j=1, \cdots ,4$ in the error-norm representation \eqref{s1-4} separately. By taking 
$\bar{\boldsymbol{\theta}}$ as the average of $\boldsymbol{\theta}$ for each $K \in \mathcal{T}_h$ we get
\begin{align*}
\mathbb{S}_1 = &- \bint{\alpha \boldsymbol{e}_q}{\boldsymbol{\theta} - \Piv \boldsymbol{\theta}} 
               + \bint{\alpha \boldsymbol{\delta}_q}{ \Piv \boldsymbol{\theta}} \\
= &- \bint{\alpha \boldsymbol{e}_q}{\boldsymbol{\theta} - \Piv \boldsymbol{\theta}} 
       - \bint{\alpha \boldsymbol{\delta}_q}{ \boldsymbol{\theta} - \Piv \boldsymbol{\theta}} 
              + \bint{\alpha \boldsymbol{\delta}_q}{ \boldsymbol{\theta} - \bar{\boldsymbol{\theta}}} \\
\le & \; C_{\alpha} \|\boldsymbol{e}_q\|_{\alpha, \Omega} \|\boldsymbol{\theta} - \Piv \boldsymbol{\theta}\|_0 
      + C_{\alpha}\|\boldsymbol{\delta}_q\|_{\alpha,\Omega} \|\boldsymbol{\theta} 
          - \Piv \boldsymbol{\theta}\|_0 + C_{\alpha}\|\boldsymbol{\delta}_q\|_{\alpha, \Omega} \|\boldsymbol{\theta} - \bar{\boldsymbol{\theta}}\|_0 \\
\le & \;  C_{\alpha} h (\|\boldsymbol{e}_q\|_{\alpha, \Omega} + \|\boldsymbol{\delta}_q\|_{\alpha, \Omega} \|\boldsymbol{\theta}\|_1  
\le \;  C_{\alpha} h (\|\boldsymbol{e}_q\|_{\alpha, \Omega} + \|\boldsymbol{\delta}_q\|_{\alpha, \Omega}) \|e_u\|_0.   &&\text{by \eqref{regularity}}
\end{align*}

Next we consider the term $ \mathbb{S}_2$. First using the fact that 
$(e_u - \ehatu)|_F, \boldsymbol{e}_q \cdot \n - \ehatq \cdot \n|_F \in M_h(F), $ for all 
$F \in \partial \mathcal{T}_h, F \cap \mathcal{E}_H = \emptyset$
we transform the integrals over the boundaries of the elements of the fine mesh into 
integrals on the boundaries on the coarse mesh only:
\begin{align*}
\mathbb{S}_2 = & \;\bintEh{e_u - \ehatu}{\boldsymbol{\theta} \cdot \n - P^{\partial}_h(\boldsymbol{\theta} \cdot \n)} 
        - \bintEh{\boldsymbol{e}_q \cdot \n - \ehatq \cdot \n}{\phi - P^{\partial}_h \phi} \\
 = &\;\bintEH{e_u - \ehatu}{\boldsymbol{\theta} \cdot \n - P^{\partial}_h(\boldsymbol{\theta} \cdot \n)} 
        - \bintEH{\boldsymbol{e}_q \cdot \n - \ehatq \cdot \n}{\phi - P^{\partial}_h \phi}.
\\
\intertext{Further, using the approximation property 
\eqref{dual_trace_1}-\eqref{dual_trace_2} and the regularity assumption \eqref{regularity} we get}
\mathbb{S}_2 \le & \;\|e_u - \ehatu\|_{\tau, \partial \mathcal{T}_h} \tau^{-\frac{1}{2}} \|
\boldsymbol{\theta} \cdot \n - P^{\partial}_h(\boldsymbol{\theta} \cdot \n)\|_{\partial \mathcal{T}_H} 
       + \|\boldsymbol{e}_q \cdot \n - \ehatq \cdot \n\|_{\partial \mathcal{T}_H}\| \phi - P^{\partial}_h \phi\|_{\partial \mathcal{T}_H} \\
\le & \; C h^{\frac{1}{2}} \tau^{-\frac{1}{2}}  \|e_u - \ehatu\|_{\tau, \partial \mathcal{T}_h} \|\boldsymbol{\theta}\|_1 
        + C h^{\frac{3}{2}} \|\boldsymbol{e}_q \cdot \n - \ehatq \cdot \n\|_{\partial \mathcal{T}_H}\| \phi\|_{2}\\
\le & \; C h^{\frac{1}{2}} \tau^{-\frac{1}{2}} \|e_u - \ehatu\|_{\tau, \partial \mathcal{T}_h} \|e_u\|_0 
        + C h^{\frac{3}{2}} \|\boldsymbol{e}_q \cdot \n - \ehatq \cdot \n\|_{\partial \mathcal{T}_H}\| e_u\|_0.
\end{align*}

Now we consider $ \mathbb{S}_3$. Using  
 $\bintEH{(I - P^{\partial}_H)(\boldsymbol{q} \cdot \n)}{ \phi} = \bintEH{(I - P^{\partial}_H) u }{ \boldsymbol{\theta} \cdot \n}=0$,
the approximation properties \eqref{PHd}, \eqref{dual_trace_1}, \eqref{dual_trace_2},  and  regularity assumption \eqref{regularity}
we get
\begin{align*}
\mathbb{S}_3 
& = -\bintEh{(I - P_M)(\boldsymbol{q} \cdot \n)}{P^{\partial}_h \phi} 
- \bintEh{(I - P_M) u }{P^{\partial}_h (\boldsymbol{\theta} \cdot \n} \\
& = \bintEH{(I - P^{\partial}_H)(\boldsymbol{q} \cdot \n)}{ \phi - P^{\partial}_h \phi} 
- \bintEH{(I - P^{\partial}_H) u }{ \boldsymbol{\theta} \cdot \n - P^{\partial}_h (\boldsymbol{\theta} \cdot \n)} \\
& \le C H^{t} h^{\frac{3}{2}} L^{-\frac{1}{2}} \|\boldsymbol{q}\|_{t+1} \|\phi\|_2 
     + C H^{t} h^{\frac{1}{2}}L^{-\frac{1}{2}} \|u\|_{t+1} \|\boldsymbol{\theta}\|_1 \\
& \le C H^{t} h^{\frac{3}{2}} L^{-\frac{1}{2}}\|\boldsymbol{q}\|_{t+1} \|e_u\|_0 
     + C H^{t} h^{\frac{1}{2}} L^{-\frac{1}{2}}\|u\|_{t+ 1} \|e_u\|_0,
\end{align*}
for any $0 \le t \le l+1$.

Next, we estimate the last term. 
 Using \eqref{dual_trace_3}, \eqref{dual_trace_4} and the regularity assumption \eqref{regularity}, we get
\begin{align*}
\mathbb{S}_4 &= - \bintEH{\boldsymbol{e}_q \cdot \n}{\phi - \mathcal{I}_H^0 \phi} 
            + \bintEH{\boldsymbol{e}_q \cdot \n - \ehatq \cdot \n}{\phi - \mathcal{I}_H^0 \phi} \\
& \le \sum_{T\in \mathcal{P}}\|\phi - \mathcal{I}_H^0 \phi\|_{\frac{1}{2}, \partial T} (\|\boldsymbol{e}_q\|_{0,T} 
        + \|\nabla \cdot \boldsymbol{e}_q\|_{0,T}) + \|\boldsymbol{e}_q \cdot \n - \ehatq \cdot \n\|_{\partial \mathcal{T}_H} \|\phi 
           - \mathcal{I}_H^0 \phi\|_{\partial \mathcal{T}_H} \\
& \le \sum_{T\in \mathcal{P}}\|\phi - \mathcal{I}_H^0 \phi\|_{\frac{1}{2}, \partial T} \|\boldsymbol{e}_q\|_{0,T} 
        + \|\boldsymbol{e}_q \cdot \n - \ehatq \cdot \n\|_{\partial \mathcal{T}_H} \|\phi - \mathcal{I}_H^0 \phi\|_{\partial \mathcal{T}_H}  
           \quad \text{by Lemma \ref{H_div_conforming},} \\
& \le CH  \left ( \|\boldsymbol{e}_q\|_0  + H^{\frac{1}{2}}\|\boldsymbol{e}_q \cdot \n - \ehatq \cdot \n\|_{\partial \mathcal{T}_h} \right ) \|\phi\|_2\\
& \le CH  \left ( \|\boldsymbol{e}_q\|_0 + H^{\frac{1}{2}}\|\boldsymbol{e}_q \cdot \n - \ehatq \cdot \n\|_{\partial \mathcal{T}_h} \right )\|e_u\|_0.
\end{align*}

Finally, we show how to bound $\|\boldsymbol{e}_q \cdot \n - \ehatq \cdot \n\|_{\partial \mathcal{T}_h} $. 
By Lemma \ref{H_div_conforming}, \ref{bdy_to_interior} and \eqref{ehatq}, we have
\begin{align*}
\|\boldsymbol{e}_q \cdot \n - \ehatq \cdot \n\|_{\partial \mathcal{T}_h} & 
           =\|\boldsymbol{e}_q \cdot \n - \ehatq \cdot \n\|_{\partial \mathcal{T}_H} 
           = \|\tau(e_u -\ehatu) - (P^{\partial}_h - P_M)(\boldsymbol{q} \cdot \n + \tau u)\|_{\partial \mathcal{T}_H}\\
& \le \tau^{\frac{1}{2}}\|e_u - \ehatu\|_{\tau, \partial \mathcal{T}_H} + 
          \|(P^{\partial}_h - P^{\partial}_H)(\boldsymbol{q} \cdot \n + \tau u)\|_{\partial \mathcal{T}_H}\\
& \le \tau^{\frac{1}{2}}\|e_u - \ehatu\|_{\tau, \partial \mathcal{T}_H} + C h^s L^{-\frac{1}{2}}(\|\boldsymbol{q}\|_{s+1} 
            + \tau\|u\|_{s+ 1}) + C H^t L^{-\frac{1}{2}}(\|\boldsymbol{q}\|_{t +1} + \tau\|u\|_{t + 1}),
\end{align*}
for all $1 \le s \le k + 1; 0 \le t \le l+1$. 

Combining the estimates for $\mathbb{S}_1 $ -- $ \mathbb{S}_4$  and  grouping the similar terms we get
\begin{align*}
\|e_u\|_0  \le & \; (C_{\alpha}h+C H)\|\boldsymbol{e}_q\|_{\alpha, \Omega} + C_{\alpha}h \|\boldsymbol{\delta}_q\|_{\alpha, \Omega}
       + C(h^{\frac{1}{2}} \tau^{-\frac{1}{2}} 
              +\tau^{\frac{1}{2}}  H^{\frac{3}{2}} )\|e_u - \ehatu\|_{\tau,\partial \mathcal{T}_h} \\
& + C H^{\frac{3}{2}}L^{-\frac{1}{2}}h^s (\|\boldsymbol{q}\|_{s+ 1} + \tau \|u\|_{s+ 1}) 
+ C H^t L^{-\frac{1}{2}} (H^{\frac{3}{2}}\|\boldsymbol{q}\|_{t+1} + (h^{\frac{1}{2}} + \tau H^{\frac{3}{2}})\|u\|_{t+1} ),
\end{align*}
for all $1 < s \le k+1, 0 \le t \le l+1$.
\end{proof}

\section{Multiscale HDG methods}\label{sec:MS}

In this section, we will consider the problem involving multiscale features. Namely, let us assume that the \emph{permeability} 
coefficient $\alpha$ has two separated scales, 
\begin{equation}\label{separation}
\alpha(x) = \alpha(x, x/\epsilon),
\end{equation}
where $x$ is called the slowly varying variable and $x/\epsilon$ is called the fast varying variable. 
Under this assumption, the exact solution $u, \boldsymbol{q}$ also
has two scales. Therefore, the derivatives of $u, \boldsymbol{q}$ also depends on the small scale $\epsilon$. 
In fact, the exact solution $(u, \boldsymbol{q})$ asymptotically behaves (\cite{JKO94}) as
\[
\|D^k u\|_{\Omega} = \mathcal{O}(\epsilon^{-(k-1)}), \quad \|D^k \boldsymbol{q}\|_{\Omega} = \mathcal{O}(\epsilon^{-k}),
\]
for all $k \ge 1$. Here $D^ku$ and $D^k \boldsymbol{q} $ denote any $k$-th partial derivative of $ u$ and $ \boldsymbol{q}$, respectively.
Then, if we set the nonzero $\tau$ to be constant 1, by Theorem \ref{estimate_q}, the velocity error becomes:
\[
\|\boldsymbol{q} - \boldsymbol{q}_h\|_{\Omega} \le C \left[ \left (\frac{h}{\epsilon} \right )^{s} 
         +  \left (\frac{H}{\epsilon}\right )^{t - \frac{1}{2}} L^{-\frac{1}{2}} \epsilon^{-\frac12} 
         +  \left (\frac{h}{\epsilon}\right )^s L^{-\frac{1}{2}} \epsilon^{-1} \right].
\]

For a multiscale finite element method, the relation between all scales should be $h < \epsilon < H < L$. 
The above estimate is no longer valuable since $\frac{H}{\epsilon} > 1$. The error for $u$ also has similar problem.
In fact, this is a typical drawback for methods using polynomials for both fine and coarse scales.
 If we look at the estimate carefully, we can see the trouble appears on 
the term $\frac{H}{\epsilon}$ only. The scale $H$ is solely associated with the coarse space $M_H$. 
This suggests that we should define the space $M_H$ in a more appropriate way so that its approximation property 
is independent of the scale $\epsilon$. This reasoning has been used by Arbogast and Xiao \cite{Arbogast_Xiao_2013}
to design a mortar multiscale finite element method that overcomes this deficiency of the standard muliscale method.
Their construction is based on the idea of involving the three scales we have used in our considerations.
However, instead of using mortar spaces to glue the approximations on the coarse grid, here we use the
same mechanism that is provided by the hybridization of the discontinuous Galerkin method.

\subsection{Homogenization results}\label{sec:homo}

In a very special case of periodic arrangement of the heterogeneous coefficient we propose to use 
non-polynomial spaces for the Lagrange multipliers that are based on 
the concept of existence of smooth solution of a homogenized problem and using the first order correction from the homogenization theory.

We first review some classical homogenization results. For more details, 
we refer readers to \cite{JKO94,EfendievHou_MSFEM_book}. 
We assume that $\alpha(x, y)$ is periodic in $y$ with the unite cell $\boldsymbol{Y} = [0,1]^d$ as its period. 
The homogenized problem is defined as
\begin{subequations}\label{homogenized_equation}
\begin{alignat}{2}
\alpha_0 \nabla u_0 + \boldsymbol{q}_0 & = 0 \quad && \text{in} \; \Omega, \\
\nabla \cdot \boldsymbol{q}_0 & = f && \text{in} \; \Omega, \\
u_0 &= 0 && \text{on} \; \partial \Omega.
\end{alignat}
\end{subequations}

Here the homogenized tensor $\alpha_0$ is defined as
\[
\alpha_0^{ij} = \int_{\boldsymbol{Y}} \alpha(x,y) (\delta_{ij} + \frac{\partial \chi_j}{\partial y_i})dy, \quad i,j = 1,2, \dots, d.
\]
Here $\chi_j, j = 1,2, \dots, d$ are the periodic
 solutions of the following cell problems: 
\[
\nabla_y \cdot [ \alpha(x,y) (\nabla_y \chi_k(x,y) + \boldsymbol{e}_k)] = 0, \quad \text{in} \; \Omega \times \boldsymbol{Y}, \;\; k = 1,2, \dots, d,
\]
where $\boldsymbol{e}_k$ is the standard unit vector in $\mathbb{R}^d$. 
Then for the first order corrector
$$
u_{\epsilon} := u_0 + \epsilon \boldsymbol{\chi} \cdot \nabla u_0
$$
 with
$\boldsymbol{\chi} = (\chi_1,  \dots, \chi_d)$ we have the following result:
\begin{lemma}\label{homogenization_approximation}
If $u_0 \in H^2(\Om) \cap H^1_0(\Omega)$, then there is some constant $C$ independent of $\epsilon$, such that
\[
\|u - u_{\epsilon}\|_0 \le C \epsilon \|u_0\|_{2}.
\]
Moreover, if $u_0 \in H^2(\Om) \cap H^1_0(\Omega) \cap W^{1, \infty}(\Omega)$, then we have (e.g., \cite{JKO94})
\[
\|\nabla(u - u_{\epsilon})\|_0 \le C (\epsilon \|u_0\|_2 + \sqrt{\epsilon}\|\nabla u_0\|_{\infty}).
\]
\end{lemma}

\subsection{A multiscale coarse space $M_H$}

Using the above basic results form homogenization of heterogeneous differential operators we shall design our 
multiscale method. As before, we introduce finite element partitioning of the domain. In this setting we 
assume that the partitions are such that
\begin{equation}\label{scales}
h < \epsilon << H < L \le 1.
\end{equation}
In this section we shall use the same polynomial spaces ${\bf V}(K)$ and $W(K)$ as before. The difference will be in the
choice of the coarse space $M_H$. Here we shall follow the work of Arbogast and Xiao \cite{Arbogast_Xiao_2013},
where this construction was used for the mortar finite element method.

For each $F \in \mathcal{E}_H$, let $\bar{F}$ denotes a rectangular neighborhood of $F$, we define the local space as
\begin{equation}\label{Lagrange-MS}
M_H(F) := \{ \mu \in L^2(F)\; : \; \mu = (1 + \epsilon \boldsymbol{\chi} \cdot \nabla) p |_F, \text{for} \; p \in P^l(\bar{F})\}.
\end{equation}
Notice that, the local space $M_H(F)$ involves both, local cell solutions $  \boldsymbol{\chi} $
 and polynomial space $P(\bar{F})$. Therefore, its dimension is larger than $P^l(F)$. 
Simple considerations show that  its dimension will depend on the structure of $ \boldsymbol{\chi} $  and will between $2l$  and $3l$.

The coarse space $M_H$ is then defined as:
\[
M_H := \{ \mu \in L^2(\mathcal{E}_{h,H}) \; : \; \text{for} \; F \in \mathcal{E}_H, 
                 \; \mu|_F \in M_H(F), \quad \text{and} \; \mu|_{\mathcal{E}^0_h \cup \partial \Omega} = 0\} 
\]
On each $F \in \mathcal{E}_H$, we define the following projection of $u_\epsilon$ on $M_H $:
\[
x \in F \in \mathcal{E}_H: \qquad
\mathcal{I} u_{\epsilon}(x) = (1 + \epsilon \boldsymbol{\chi} \cdot \nabla) \mathcal{I}_H u_0(x).
\]
Here $\mathcal{I}_H u_0$ is defined on $\bar F$ as the orthogonal 
$L^2$-projection of $u_0$ into $P^l(\bar{F})$. It has the following standard approximation property
for $1 \le t \le l+1$ and $ F \in  \mathcal{E}_H$:
\[
\|u_0 - \mathcal{I}_H u_0\|_{r,\bar F} \le C H^{t - r} \|u_0\|_{t,\bar F},
\]
for all $0 \le r \le t \le l+1$.
Also, for any function $\xi \in H^1(\bar{F})$, we have the following two trace inequalities:
$$ 
\|\phi\|_{0, F} \le H^{-\frac12}\|\xi\|_{0,\bar F} + H^{\frac12}\|\nabla \xi\|_{0, \bar F}, 
\quad  \|\phi\|_{\frac12, F}  \le H^{-1} \|\xi\|_{0, \bar F} + \|\nabla \xi\|_{0, \bar F}.
$$ 
The above two inequalities can be obtained by a simple scaling argument. Combining the trace 
inequalities and the approximation property of the interpolation, we get the following estimates:
\begin{subequations}\label{standard_6}
\begin{alignat}{1}
\|u_0 - \mathcal{I}_H u_0\|_{r, F} &\le C H^{t -\frac12 - r} \|u_0\|_{t,\bar F},   \label{standard_6a}  \\
\|\nabla( u_0 - \mathcal{I}_H u_0)\|_{r,F} &\le C H^{t-\frac32 - r}\|u_0\|_{t,\bar F}, \label{standard_6b} 
\end{alignat}
\end{subequations}

After summation over the faces $ F \in  \partial \mathcal{T}_H$  these estimates produce the following bounds:
\begin{subequations}\label{standard_7}
\begin{alignat}{1}
\|u_0 - \mathcal{I}_H u_0\|_{r, \partial \mathcal{T}_H} &\le C H^{t -\frac12 - r} \|u_0\|_{t},   \label{standard_7a}  \\
\|\nabla( u_0 - \mathcal{I}_H u_0)\|_{r, \partial \mathcal{T}_H} &\le C H^{t-\frac32 - r}\|u_0\|_{t}, \label{standard_7b} 
\end{alignat}
\end{subequations}
for all $1 \le t \le l+1, r = 0, \frac12$.

We have the following approximation result:
\begin{lemma}\label{L2Interpolant}
Let  \eqref{separation} hold and let the homogenized solution be sufficiently
smooth, i.e. $u_0 \in H^t(\Omega)$. Then for $1  \le t \le l+1$
\[
\|u - P^H_{\partial} u\|_{\partial \mathcal{T}_H}   \le \|u - \mathcal{I} u_{\epsilon}\|_{\partial \mathcal{T}_H} 
        \le C (\epsilon  L^{-\frac{1}{2}}\|u_0\|_2 + \sqrt{\epsilon} L^{-\frac{1}{2}}\|\nabla u_0\|_{\infty} +  H^{t - \frac12} \|u_0\|_{t}).
\]
\end{lemma}
\begin{proof}
The bound $\|u - P^H_{\partial} u\|_{\partial \mathcal{T}_H}  \le \|u - \mathcal{I} u_{\epsilon}\|_{\partial \mathcal{T}_H}$
is obvious, since $  P^H_{\partial}$ is an orthogonal projection on $M_H$.
Then by the  triangle inequality, we have
\[
\|u - \mathcal{I} u_{\epsilon}\|_{\partial \mathcal{T}_H} \le \|u - u_{\epsilon}\|_{\partial \mathcal{T}_H} 
                  + \|u_{\epsilon} - \mathcal{I} u_{\epsilon}\|_{\partial \mathcal{T}_H}. 
\]
Now we estimate the two terms on the right hand side separately. By the trace inequality \eqref{eq:trace} and  
 the approximation property of homogenized solution established in  Lemma \ref{homogenization_approximation}, we have
\begin{align*}
\|u - u_{\epsilon}\|_{\partial \mathcal{T}_H}
        \le CL^{-\frac{1}{2}} \|u - u_{\epsilon}\|_{1, \Omega} 
& \le C L^{-\frac{1}{2}} (\epsilon \|u_0\|_2 + \sqrt{\epsilon}\|\nabla u_0\|_{\infty}).
\end{align*}
The second term is bounded by using the approximation property \eqref{standard_7} in the following manner:
\begin{align*}
\|u_{\epsilon} - \mathcal{I} u_{\epsilon}\|_{\partial \mathcal{T}_H} &= \|(1 + \epsilon \boldsymbol{\chi} \cdot \nabla) (u_0 - \mathcal{I}_H u_0)\|_{\partial \mathcal{T}_H} \\
& \le  \| u_0 - \mathcal{I}_H u_0\|_{\partial \mathcal{T}_H} +  \epsilon \|\boldsymbol{\chi}\|_{\infty} \|\nabla(u_0 - \mathcal{I}_H u_0)\|_{\partial \mathcal{T}_H} \\
& \le C H^{t-\frac12} \|u_0\|_{t} + C \epsilon H^{t-\frac32} \|u_0\|_{t}   \le C H^{t-\frac12} \|u_0\|_{t},
\end{align*}
for $1 \le t \le l+1$.
This completes the proof. \hfill \end{proof}

We can prove  a better estimate for smoother $u_0$. To get such a result we need an additional assumption on the space $M_H$:
\begin{assumption}\label{C1_conforming}
The space $P^l(\bar{\mathcal{E}}_H) = \cup_{F \in \mathcal{E}_H} P^l(\bar{F})$ has a subspace which provides 
an approximation of order $\mathcal{O}(H^{l+1})$ for smooth $u_0$ and the restriction on 
$\mathcal{E}_H: P^l(\bar{\mathcal{E}}_H)|_{\mathcal{E}_H} \subset M_H$ is $C^1-$conforming over the coarse skeleton $\mathcal{E}_H$.
\end{assumption}

If this assumption holds, we can define the interpolation of $u_{\epsilon}$ to be:
\[
\mathcal{I}^1 u_{\epsilon} := (1 + \epsilon \boldsymbol{\chi} \cdot \nabla) \mathcal{I}^1_H u_0,
\]
where $\mathcal{I}^1_H u_0$ is the $C^1$-interpolation of $u_0$ onto $P^l(\bar{\mathcal{E}}_H)$. 
Under this assumption, the interpolation $\mathcal{I}^1_H u_0$ has the same approximation property \eqref{standard_7} as $\mathcal{I}_H u_0$.

In order to improve our estimates, we will need the following approximation result, 
\begin{lemma}\label{interpolation_approximation}
In addition to Assumption \ref{C1_conforming}, assume  the  homogenized solution of \eqref{homogenized_equation} $u_0 $ belongs to 
$ H^{l+1}(\Omega) \cap H^2(\Omega) \cap W^{1, \infty}(\Omega)$.  Then on  each subdomain 
$T \subset \Omega$ and $1 \le t \le l+1$   we have
\[
\langle u - \mathcal{I}^1  u_{\epsilon} \,,\, \boldsymbol{e}_q \cdot \n \rangle_{\partial T} 
\le C (H^{t-2} \|u_0\|_{t, T} + \epsilon \|u_0\|_{2,T} + \sqrt{\epsilon}\|\nabla u_0\|_{\infty, T}) \|\boldsymbol{e}_q\|_{T}.
\]
\end{lemma}
\begin{proof}
Similar as in the proof of Lemma \ref{L2Interpolant}, we begin by splitting the term as:
\begin{align*}
\langle u - \mathcal{I}^1  u_{\epsilon} \,,\, \boldsymbol{e}_q \cdot \n \rangle_{\partial T} & 
      = \langle u - u_{\epsilon} \,,\, \boldsymbol{e}_q \cdot \n \rangle_{\partial T} + \langle u_{\epsilon} - \mathcal{I}^1  u_{\epsilon} 
          \,,\, \boldsymbol{e}_q \cdot \n \rangle_{\partial T}.
\end{align*}
For the first term, we apply the Stokes' Theorem:
\begin{align*}
\langle u - u_{\epsilon} \,,\, \boldsymbol{e}_q \cdot \n \rangle_{\partial T} 
       &= \|u - u_{\epsilon}\|_{1, T} \|\boldsymbol{e}_q \|_{H(div, T)} = \|u - u_{\epsilon}\|_{1, T} \|\boldsymbol{e}_q\|_{T} \quad 
      &&\text{by Lemma \ref{H_div_conforming},} \\
& \le C (\epsilon \|u_0\|_{2,T} + \sqrt{\epsilon} \|\nabla u_0\|_{{\infty}, T})  \|\boldsymbol{e}_q\|_{T} 
         \quad &&\text{by Lemma \ref{homogenization_approximation}.}
\end{align*}
For the second term, by the definition of the interpolation $\mathcal{I}^1 u_{\epsilon}$, we have:
\begin{align*}
|\langle u_{\epsilon} - \mathcal{I}^1  u_{\epsilon} \,,\, \boldsymbol{e}_q \cdot \n \rangle_{\partial T}| \le 
                        & |\langle u_{0} - \mathcal{I}^1_H  u_{0} \,,\, \boldsymbol{e}_q \cdot \n \rangle_{\partial T}| 
                            +  \epsilon \, | \langle \nabla( u_{0} - \mathcal{I}^1_H  u_{0} )
                         \,,\, \boldsymbol{\chi} \otimes \boldsymbol{e}_q \cdot \n \rangle_{\partial T}| \\
\intertext{due to Assumption \ref{C1_conforming}, we have $u_0 - \mathcal{I}^c_H u_0 \in H^{\frac23}(\partial T)$, 
so we can apply Stokes' Theorem on both terms and obtain}
|\langle u_{\epsilon} - \mathcal{I}^1  u_{\epsilon} \,,\, \boldsymbol{e}_q \cdot \n \rangle_{\partial T}| 
&\le \|u_0 - \mathcal{I}^1_H u_0\|_{\frac12, \partial T} \|\boldsymbol{e}_q\|_{H(div), T} 
                 +  \epsilon \, \|\nabla( u_{0} - \mathcal{I}^1_H  u_{0})\|_{\frac12, \partial T} \|\boldsymbol{\chi} \otimes \boldsymbol{e}_q\|_{H(div), T} \\ 
&\le \|u_0 - \mathcal{I}^1_H u_0\|_{\frac12, \partial T} \|\boldsymbol{e}_q\|_{T} \\
& \hspace{1cm}+ \epsilon \, \|\nabla( u_{0} - \mathcal{I}^1_H  u_{0})\|_{\frac12, \partial T} (\|\boldsymbol{\chi}\|_{\infty, T} \|\boldsymbol{e}_q\|_{T} 
                + \|\nabla \boldsymbol{\chi}\|_{\infty, T} \|\boldsymbol{e}_q\|_{T}),
\end{align*}
here we used the fact that $\|\nabla \cdot \boldsymbol{e}_q\|_{T} = 0$. Finally, under the assumption \eqref{separation}, 
we have the classic result: 
$\|\boldsymbol{\chi}\|_{\infty, T} = \mathcal{O}(1), \|\nabla \boldsymbol{\chi}\|_{\infty, T} = \mathcal{O}(\epsilon^{-1})$, 
see \cite{JKO94}. Applying these results and the approximation result \eqref{standard_7}, we have
\begin{align*}
|\langle u_{\epsilon} - \mathcal{I}^1  u_{\epsilon} \,,\, \boldsymbol{e}_q \cdot \n \rangle_{\partial T}|
& \le C H^{t-1} \|u_0\|_{t, T}  \|\boldsymbol{e}_q\|_{0, T} +C \epsilon H^{t-2} \|u_0\|_{t, T} (1 + \epsilon^{-1})\|\boldsymbol{e}_q\|_T  \\
& \le C H^{t-2}\|u_0\|_{t, T} \|\boldsymbol{e}_q\|_T,
\end{align*}
which completes the proof. \hfill 
\end{proof}

\subsection{Estimate for $\boldsymbol{q} - \boldsymbol{q}_h$}

We are now ready to establish the following result. 
\begin{theorem}\label{multiscale_estimate_q}
Let the coefficient $\alpha$ satisfy \eqref{separation}. Then there exists a constant $C$, independent of $h,H,L$ and $\epsilon$, such that
\begin{equation}\label{error_usual}
\begin{aligned}
\|\boldsymbol{e}_q\|_{\alpha, \Omega} + \|e_u - \ehatu\|_{\tau, \mathcal{T}_h} & 
    \le  C \|\boldsymbol{q} - \Piv \boldsymbol{q}\|_{\alpha, \Omega} + C h^s L^{-\frac12}( 
 \tau^{-\frac{1}{2}}\|\boldsymbol{q}\|_{s+1} + \tau^{\frac{1}{2}}\|u\|_{s+1}) \\
& + C(1 + \tau^{\frac{1}{2}}L^{-\frac{1}{2}})(\epsilon \|u_0\|_2 + \sqrt{\epsilon} \|\nabla u_0\|_{\infty}) \\
& + C (\tau^{\frac{1}{2}} + h^{-\frac{1}{2}}) H^{t-\frac12} \|u_0\|_{t},
\end{aligned}
\end{equation}
for all $1 \le s \le k+1, \, 1 \le t \le l+1$. 

Moreover, if Assumption \ref{C1_conforming} holds, then the following  improved estimate holds:
\begin{equation}\label{error_better}
\begin{aligned}
\|\boldsymbol{e}_q\|_{\alpha,\Omega} + \|e_u - \ehatu\|_{\tau, \mathcal{T}_h} &
        \le C \|\boldsymbol{q} - \Piv \boldsymbol{q}\|_{\alpha, \Omega} +C h^s L^{-\frac12}( 
        \tau^{-\frac{1}{2}}\|\boldsymbol{q}\|_{s+1} + \tau^{\frac{1}{2}}\|u\|_{s+1}) \\
& + C(1 + \tau^{\frac{1}{2}}L^{-\frac{1}{2}})(\epsilon \|u_0\|_2 + \sqrt{\epsilon} \|\nabla u_0\|_{\infty}) \\
& + C (\tau^{\frac{1}{2}} + H^{-\frac32}) H^{t-\frac12} \|u_0\|_{t}, 
\end{aligned}
\end{equation}
for all $1 \le s \le k+1, \, 1 \le t \le l+1$. 
\end{theorem}
\begin{proof}
By Lemma \ref{energy_argument} we have
\begin{alignat*}{1}
\|\boldsymbol{e}_q\|^2_{\alpha, \Omega} + \|e_u- \ehatu\|^2_{\alphaa, \partial \mathcal{T}_h}
& = -\bint{\alpha \boldsymbol{\delta}_q}{\boldsymbol{e}_q} - \bintEH{(I - P^H_{\partial}) u}{\boldsymbol{e}_q \cdot \n} \\
& + \bintEH{P^h_{\partial} u  - P^H_{\partial} u}{\tau(e_u - \ehatu)} - \bintEH{(I - P^h_{\partial})(\boldsymbol{q}\cdot \n)}{e_u - \ehatu}\\
& = T_1  + T_2 + T_3 +  T_4.
\end{alignat*}
We will estimate these four terms separately. For the first term we easily get
\[
|T_1 | \le \|\alpha\|_{\infty} \|\boldsymbol{\delta}_q\|_{\alpha, \Omega} \|\boldsymbol{e}_q\|_{\alpha, \Omega} 
\quad \mbox{for all} \quad 1\le s \le k+1.
\]

Further, using the approximation properties of  $P^h_{\partial}$ 
and $ P^H_{\partial}$ established in Lemma \ref{bdy_to_interior} we get
\begin{align*}
|T_3| &=  |\bintEH{P^h_{\partial} u  - P^H_{\partial} u}{\tau(e_u - \ehatu)}|  \\
&  =  | \bintEH{u - P^h_{\partial} u }{\tau(e_u - \ehatu)} - \bintEH{ u  - P^H_{\partial} u}{\tau(e_u - \ehatu)}| \\
&  \le \tau^{\frac{1}{2}}( \|u - P^h_{\partial} u\|_{\partial \mathcal{T}_H} 
          + \|u - P^H_{\partial} u\|_{\partial \mathcal{T}_H}) \|e_u- \ehatu\|_{\alphaa, \partial \mathcal{T}_h} \\
&\le C \tau^{\frac{1}{2}}(L^{-\frac12}h^s \|u\|_{s+1} 
         + \|u - P^H_{\partial} u\|_{\partial \mathcal{T}_H}) \|e_u- \ehatu\|_{\alphaa, \partial \mathcal{T}_h}
                           \quad  \quad \text{by \eqref{Phd}} 
\end{align*}
and bound the term $  \|u - P^H_{\partial} u\|_{\partial \mathcal{T}_H}$ using Lemma \ref{L2Interpolant}.

In a similar way we estimate $T_4$:
\[
|T_4| \le \tau^{-\frac{1}{2}} \|(I - P^h_{\partial})(\boldsymbol{q}\cdot \n)\|_{\partial \mathcal{T}_H} \|e_u - \ehatu\|_{\tau,\mathcal{T}_H} 
           \le C \tau^{-\frac{1}{2}}L^{-\frac{1}{2}}h^s \|\boldsymbol{q}\|_{s+1}\|e_u - \ehatu\|_{\tau,\mathcal{T}_h}.
\]
Finally, for $T_2$, if we simply apply Cauchy-Schwarz inequality and trace inequality, we will have a term of order $\sqrt{\frac{\epsilon}{h}}$. 
This estimate is not desirable since $h < \epsilon$. To bound $T_2$ in a better way, we first 
rewrite it  using the error equation \eqref{error-equation-3}:
\begin{align*}
T_2 &= - \bintEH{u  - \mathcal{I} u_{\epsilon}}{\boldsymbol{e}_q \cdot \n} 
         -  \bintEH{\mathcal{I} u_{\epsilon} - P^H_{\partial} u}{\boldsymbol{e}_q \cdot \n} \\
&=  - \bintEH{u - \mathcal{I} u_{\epsilon}}{\boldsymbol{e}_q \cdot \n} 
         -  \bintEH{\mathcal{I} u_{\epsilon} - P^H_{\partial} u}{\boldsymbol{e}_q \cdot \n - \ehatq \cdot \n}. 
\end{align*}
Then using the equation \eqref{ehatq} we get
\begin{align*}
T_2 &= - \bintEH{u- \mathcal{I} u_{\epsilon}}{\boldsymbol{e}_q \cdot \n} -  \bintEH{\mathcal{I} u_{\epsilon} - P^H_{\partial} u}{\tau(e_u - \ehatu)} \\
&+  \bintEH{\mathcal{I} u_{\epsilon} - P^H_{\partial} u}{(P^h_{\partial} - P^H_{\partial})(\boldsymbol{q} \cdot \n + \tau u)} : = T_{21} + T_{22} + T_{23}.
\end{align*}
Now we estimate these three terms separately. Now recall that  $\|\nabla \cdot \boldsymbol{e}_q\|_T=0 $  
for $T \in \mathcal{P}$ by Lemma \ref{H_div_conforming}. Thus
using divergence theorem we get
\begin{align*}
T_{21} & =  \bintEH{u- u_{\epsilon}}{\boldsymbol{e}_q \cdot \n} 
               +  \bintEH{u_{\epsilon}- \mathcal{I} u_{\epsilon}}{\boldsymbol{e}_q \cdot \n} \\
& = \sum_{T \in \mathcal{T}}(\nabla ( u- u_{\epsilon}), \boldsymbol{e}_q)_T 
               + \bintEH{u_{\epsilon}- \mathcal{I} u_{\epsilon}}{\boldsymbol{e}_q \cdot \n}\\
&\le \sum_{T \in \mathcal{T}} \|u - u_{\epsilon}\|_{1,T} \|\boldsymbol{e}_q\|_{ T} 
               + \bintEH{(1 + \epsilon \boldsymbol{\chi}\cdot \nabla)(u_0 - \mathcal{I}_H u_0)}{\boldsymbol{e}_q \cdot \n} \\
& \le  \|u - u_{\epsilon}\|_{1} \|\boldsymbol{e}_q\| + (\|u_0 - \mathcal{I}_H u_0\|_{\partial \mathcal{T}_H} 
               + \epsilon \|\boldsymbol{\chi}\|_{\infty} \|\nabla (u_0 - \mathcal{I}_H u_0)\|_{\partial \mathcal{T}_H}) \|\boldsymbol{e}_q \cdot \n\|_{\partial \mathcal{T}_H}.
\end{align*}
Now using the trace inequality \eqref{eq:trace} (for $D=T$) and \eqref{standard_7} we get
\begin{align*}
|T_{21}|  
&\le C \|\boldsymbol{e}_q\|_0 \left \{ \|u - u_{\epsilon}\|_1 + h^{-\frac{1}{2}} \big (\|u_0 - \mathcal{I}_H u_0\|_{\partial \mathcal{T}_H} 
                                         + \epsilon \|\nabla(u_0 - \mathcal{I}_H u_0)\|_{\partial \mathcal{T}_H}\big ) \right \} \\
& \le C \|\boldsymbol{e}_q\|_0 \left \{ \epsilon \|u_0\|_2 + \sqrt{\epsilon} \|\nabla u_0\|_{\infty} 
                                         + h^{-\frac{1}{2}} H^{t-\frac12}\|u_0\|_{t} \right \}.
\end{align*}
For estimating $T_{22}$ we apply Cauchy-Schwarz and triangle inequalities
\[
|T_{22} | \le \tau^{\frac{1}{2}}\|\mathcal{I} u_{\epsilon} - P^H_{\partial} u\|_{\partial \mathcal{T}_H} \|e_u - \ehatu\|_{\tau, \partial \mathcal{T}_H}  
                               \le 2 \tau^{\frac{1}{2}}\|u - \mathcal{I} u_{\epsilon}\|_{\partial \mathcal{T}_H} \|e_u - \ehatu\|_{\tau, \partial \mathcal{T}_h}
\]
and then bound $ \|u - \mathcal{I} u_{\epsilon}\|_{\partial \mathcal{T}_H}$  using  Lemma \ref{L2Interpolant}.
For $T_{23}$, we have
\begin{align*}
|T_{23}| &= |\bintEH{\mathcal{I} u_{\epsilon} - P^H_{\partial} u}{(P^h_{\partial} - I)(\boldsymbol{q} \cdot \n + \tau u)}| \\
&\le 2\|u - \mathcal{I} u_{\epsilon}\|_{\partial \mathcal{T}_H} (\|\boldsymbol{q} \cdot \n - P^h_{\partial} (\boldsymbol{q} \cdot \n)\|_{\partial \mathcal{T}_H} 
                       + \tau \|u - P^h_{\partial} u\|_{\partial \mathcal{T}_H})  \\
& \le C L^{-\frac{1}{2}}h^s (\|\boldsymbol{q}\|_{s+1} + \tau \|u\|_{s+1})\|u - \mathcal{I} u_{\epsilon}\|_{\partial \mathcal{T}_H} \\
& \le C \left\{ (h^s L^{-\frac12} (\tau^{-\frac12} \|\boldsymbol{q}\|_{s+1} + \tau^{\frac12}\|u\|_{s+1}))^2 
         + (\tau^{\frac12}\|u - \mathcal{I} u_{\epsilon}\|_{\partial \mathcal{T}_H} )^2 \right\}.
\end{align*}
Then we estimate the term $\|u - \mathcal{I} u_{\epsilon}\|_{\partial \mathcal{T}_H} $ by Lemma \ref{L2Interpolant} and
get \eqref{error_usual}  by combining the estimates for $T_1, T_2, T_3, T_4$. 

The estimate \eqref{error_better} we shall prove under the  Assumption \ref{C1_conforming}, 
namely, a $C^1$-conforming interpolate $\mathcal{I}^1_H u_{\epsilon}$ exists and it converges to 
$u_{\epsilon}$ with order $\mathcal{O}(H^{l+1})$. Then we can replace 
the $L^2$ interpolate $\mathcal{I} u_{\epsilon}$ by $\mathcal{I}^1_H u_{\epsilon}$ and 
then apply Lemma \ref{interpolation_approximation} to $T_{21}$ to get
\begin{align*}
T_{21} &= \bintEH{u- \mathcal{I}^1_H u_{\epsilon}}{\boldsymbol{e}_q \cdot \n}
\le  C \left ( H^{t - 2} \|u_0\|_{t } 
                    + \epsilon \|u_0\|_{2} + \sqrt{\epsilon}\|\nabla u_0\|_{\infty} \right ) \|\boldsymbol{e}_q\|_{0},
\end{align*}
which completes the proof.
\hfill   \end{proof}

As a consequence of Theorem \ref{multiscale_estimate_q}, we immediately obtain an $L^2$ estimate for $\boldsymbol{q} - \boldsymbol{q}_h$:
\begin{corollary}\label{L2estimate_q}
Suppose we have the same assumption as Theorem \ref{multiscale_estimate_q}.Then we have
\begin{align*}
\|\boldsymbol{q} - \boldsymbol{q}_h\|_{\alpha, \Omega}  &\le  C \|\boldsymbol{q} - \Piv \boldsymbol{q}\|_{\alpha, \Omega} + C h^s L^{-\frac12}( 
 \tau^{-\frac{1}{2}}\|\boldsymbol{q}\|_{s+1} + \tau^{\frac{1}{2}}\|u\|_{s+1}) \\
& + C(1 + \tau^{\frac{1}{2}}L^{-\frac{1}{2}})(\epsilon \|u_0\|_2 + \sqrt{\epsilon} \|\nabla u_0\|_{\infty}) 
 + C (\tau^{\frac{1}{2}} + h^{-\frac{1}{2}}) H^{t-\frac12} \|u_0\|_{t},
\end{align*}
for all $1 \le s \le k+1, \, 1 \le t \le l+1$.

Moreover, if Assumption \ref{C1_conforming} holds, then we have the following  improved  estimate
\begin{align*}
\|\boldsymbol{q} - \boldsymbol{q}_h\|_{\alpha, \Omega}  &\le  C \|\boldsymbol{q} - \Piv \boldsymbol{q}\|_{\alpha, \Omega} + C h^s L^{-\frac12}( 
 \tau^{-\frac{1}{2}}\|\boldsymbol{q}\|_{s+1} + \tau^{\frac{1}{2}}\|u\|_{s+1}) \\
& + C(1 + \tau^{\frac{1}{2}}L^{-\frac{1}{2}})(\epsilon \|u_0\|_2 + \sqrt{\epsilon} \|\nabla u_0\|_{\infty}) 
 + C (\tau^{\frac{1}{2}} + H^{-\frac32}) H^{t-\frac12} \|u_0\|_{t},
\end{align*}
for all $1 \le s \le k+1, \, 1 \le t \le l+1$. 
\end{corollary}

\subsection{Estimate for $u - u_h$} 

In Section 3, we use a standard duality argument to get an a priori estimate for $u - u_h$. 
It is based on the full $H^2$ regularity assumption \eqref{regularity} of the adjoint equation \eqref{adjoint}. 
When the permeability coefficient $\alpha$ has two separated scales, the regularity assumption is no longer valid. 
Instead, we consider the following adjoint problem:
\begin{subequations}\label{multiscale_adjoint}
\begin{alignat}{2}
\label{multiscale_adjoint_1}  \boldsymbol{\theta} + \nabla \phi & = 0 \quad && \text{in $\Omega$,}\\
\label{multiscale_adjoint_2} \nabla \cdot \boldsymbol{\theta} & = e_u && \text{in $\Omega$,} \\
\label{multiscale_adjoint_3} \phi & = 0 && \text{on $\partial \Omega$.}
\end{alignat}
\end{subequations}
We assume the above problem has full $H^2$ regularity:
\begin{equation}\label{H2_regularity}
\|\phi\|_2 + \|\boldsymbol{\theta}\|_1 \le C \|e_u\|_0,
\end{equation}
where $C$ only depends on the domain $\Omega$.

We are ready to state the estimate for $u - u_h$:
\begin{theorem}\label{multiscale_estimate_u}
Suppose the same assumptions as Theorem \ref{multiscale_estimate_q}. In addition, we also assume the 
$H^2$ regularity \eqref{H2_regularity} holds. Then we have
 \[
\|e_u\|_0 \le C \|\boldsymbol{q} - \boldsymbol{q}_h\|_{\alpha,\Omega}
     + C \Big (\frac{h}{\tau L} \Big )^{\frac{1}{2}} \|e_u - \ehatu\|_{\tau, \mathcal{T}_h} 
     + C  \Big (\frac{h}{L}  \Big )^{\frac{1}{2}} \|u - P^H_{\partial} u\|_{\partial \mathcal{T}_H}.
\]
\end{theorem}
\begin{proof}
We begin by the fact that
\begin{align*}
\|e_u\|^2_0 &= \bint{e_u}{\nabla \cdot \boldsymbol{\theta}} 
      = \bint{e_u}{\nabla \cdot \Piv \boldsymbol{\theta}} + \bint{e_u}{\nabla \cdot (\boldsymbol{\theta} - \Piv \boldsymbol{\theta}}\\
\intertext{Then integrating by parts on the second term  
and using the property \eqref{HDG  projection-2} we get} 
\|e_u\|^2_0 & = \bint{e_u}{\nabla \cdot \Piv \boldsymbol{\theta}} + \bintEh{e_u}{(\boldsymbol{\theta} - \Piv \boldsymbol{\theta})\cdot \n}. 
\end{align*}
Taking $\boldsymbol{v} = \Piv \boldsymbol{\theta}$ in the error equation \eqref{error-equation-1}, we have
\[
\bint{e_u}{\nabla \cdot \Piv \boldsymbol{\theta}} 
      = \bint{\alpha(\boldsymbol{q} - \boldsymbol{q}_h)}{\Piv \boldsymbol{\theta}} 
           + \bintEh{\ehatu}{\Piv \boldsymbol{\theta} \cdot \n} + \bintEh{(I - P_M) u}{\Piv \boldsymbol{\theta} \cdot \n}.
\] 
Now using  the fact that $\ehatu, (I - P_M)u$ are single valued on $\partial \mathcal{T}_h$, so $\bintEh{\ehatu}{\boldsymbol{\theta} \cdot \n} = \bintEh{(I - P_M)u}{\boldsymbol{\theta}\cdot \n} = 0$ after some algebraic manipulation, we obtain:
\begin{align*}
\|e_u\|^2_0 &= \bint{\alpha(\boldsymbol{q} - \boldsymbol{q}_h)}{\Piv \boldsymbol{\theta}} + \bintEh{(I - P_M)u}{\Piv \boldsymbol{\theta}\cdot \n} \\
&+ \bintEh{\ehatu}{\Piv \boldsymbol{\theta}\cdot \n} + \bintEh{e_u}{(\boldsymbol{\theta} - \Piv \boldsymbol{\theta})\cdot \n}\\
& = \bint{\alpha(\boldsymbol{q} - \boldsymbol{q}_h)}{\Piv \boldsymbol{\theta}} 
- \bintEh{(I - P_M)u}{(\boldsymbol{\theta} - \Piv \boldsymbol{\theta})\cdot \n} + \bintEh{e_u - \ehatu}{(\boldsymbol{\theta} - \Piv \boldsymbol{\theta})\cdot \n}.
\end{align*}
We now estimate the above three terms separately.
\begin{align*}
 \bint{\alpha(\boldsymbol{q} - \boldsymbol{q}_h)}{\Piv \boldsymbol{\theta}} & =  \bint{\alpha(\boldsymbol{q} - \boldsymbol{q}_h)}{ \boldsymbol{\theta}} -  \bint{\alpha(\boldsymbol{q} - \boldsymbol{q}_h)}{ \boldsymbol{\theta} - \Piv \boldsymbol{\theta}} \\
&\le C (\|\boldsymbol{q} - \boldsymbol{q}_h\|_{\alpha,\Omega} \|\boldsymbol{\theta}\|_0 + \|\boldsymbol{q} - \boldsymbol{q}_h\|_{\alpha,\Omega} \|\boldsymbol{\theta} - \Piv \boldsymbol{\theta}\|_0) \\
&\le C (\|\boldsymbol{q} - \boldsymbol{q}_h\|_{\alpha, \Omega} \|e_u\|_0 + h \|\boldsymbol{q} - \boldsymbol{q}_h\|_{\alpha,\Omega} \|\phi\|_1)\\
\intertext{so that after using the full  regularity assumption \eqref{H2_regularity} and Lemma \ref{projection approximation}, we get}
 \bint{\alpha(\boldsymbol{q} - \boldsymbol{q}_h)}{\Piv \boldsymbol{\theta}}
                             & \le C \|\boldsymbol{q} - \boldsymbol{q}_h\|_{\alpha,\Omega} \|e_u\|_0.
\end{align*}
In order to bound the other two terms, we need the following bound:
\begin{equation}\label{aux_estimate}
\|(\boldsymbol{\theta} - \Piv \boldsymbol{\theta}) \cdot \n\|_{\partial \mathcal{T}_H} \le C h^{\frac{1}{2}} L^{-\frac{1}{2}} \|e_u\|_0. 
\end{equation}
On each $T \in \mathcal{P}$, we use $\mathcal{I} \boldsymbol{\theta}$ to denote the Cl\'{e}ment interpolation from 
$H^1(T)$ into $\boldsymbol{V}(T)$. Then we have
\begin{align*}
\|(\boldsymbol{\theta} - \Piv \boldsymbol{\theta}) \cdot \n\|_{\partial \mathcal{T}_H} 
          & \le \|\mathcal{I} \boldsymbol{\theta} \cdot \n- \Piv \boldsymbol{\theta} \cdot \n\|_{\partial \mathcal{T}_H} 
               + \| (\boldsymbol{\theta} - \mathcal{I} \boldsymbol{\theta}) \cdot \n\|_{\partial \mathcal{T}_H} \\
& \le C h^{-\frac{1}{2}} \|\mathcal{I} \boldsymbol{\theta} - \Piv \boldsymbol{\theta} \|_{0} 
               + \| \boldsymbol{\theta} - \mathcal{I} \boldsymbol{\theta} \|_{\partial \mathcal{T}_H} \\
& \le C h^{-\frac{1}{2}}( \|\boldsymbol{\theta} - \Piv \boldsymbol{\theta} \|_{0} 
              +  \| \boldsymbol{\theta} - \mathcal{I} \boldsymbol{\theta}\|_{0}) + \sum_{T \in \mathcal{P}} C h^{\frac{1}{2}} \|\boldsymbol{\theta}\|_{1, T} \\
&\le C h^{\frac{1}{2}} (\|\phi\|_1 + \|\boldsymbol{\theta}\|_1) + C L^{-\frac{1}{2}} h^{\frac{1}{2}} \|\boldsymbol{\theta}\|_{1} 
 \le C h^{\frac{1}{2}} L^{-\frac{1}{2}} \|e_u\|_0
\end{align*}
the last step is due to the regularity condition \eqref{H2_regularity}. We now estimate the other two terms. First, we
have the equality
\[
\bintEh{(I - P_M)u}{(\boldsymbol{\theta} - \Piv \boldsymbol{\theta})\cdot \n}  
        = \bintEH{u - P^H_{\partial} u}{(\boldsymbol{\theta} - \Piv \boldsymbol{\theta})\cdot \n}, 
\]
due to the fact that $P_M = P^h_{\partial}$ on $F \in \mathcal{E}^0_h$ and $u - P_M u$ is single valued 
on $\partial \mathcal{T}_h$. We apply Cauchy-Schwarz inequality to obtain:
\begin{align*}
\bintEh{(I - P_M)u}{(\boldsymbol{\theta} - \Piv \boldsymbol{\theta})\cdot \n} 
    &\le \|u - P^H_{\partial}\|_{\partial \mathcal{T}_H} \| (\boldsymbol{\theta} - \Piv \boldsymbol{\theta})\cdot \n\|_{\partial \mathcal{T}_H} \\
& \le C h^{\frac{1}{2}} L^{-\frac{1}{2}} \|u - P^H_{\partial} u\|_{\partial \mathcal{T}_H}\|e_u\|_0.
\end{align*}
Finally, consider any interior edge $F \in \mathcal{E}^0_h \cap \partial \mathcal{T}_h$. If $\tau > 0$ on $F$, by the identity \eqref{local_identity}, we have
\[
\langle e_u - \ehatu \, , \, (\boldsymbol{\theta} - \Piv \boldsymbol{\theta})\cdot \n \rangle_F = 0.
\]
On the other hand, if $\tau = 0$ on $F$, by the definition of the projection \eqref{HDG projection-3}, we still have the above identity. This applies
\begin{align*}
\bintEh{e_u - \ehatu}{(\boldsymbol{\theta} - \Piv \boldsymbol{\theta})\cdot \n} &= \bintEH{e_u - \ehatu}{(\boldsymbol{\theta} - \Piv \boldsymbol{\theta})\cdot \n} \\
& \le \tau^{-\frac{1}{2}} \|e_u - \ehatu\|_{\tau, \mathcal{T}_h} \|(\boldsymbol{\theta} - \Piv \boldsymbol{\theta})\cdot \n\|_{\partial \mathcal{T}_H} \\
& \le C \tau^{-\frac{1}{2}} h^{\frac{1}{2}} L^{-\frac{1}{2}}  \|e_u - \ehatu\|_{\tau, \mathcal{T}_h} \|e_u\|_0,
\end{align*}
by the estimate \eqref{aux_estimate}. The proof is complete by combining the above three estimates.
\hfill   \end{proof}  

As a consequence of Theorem \ref{multiscale_estimate_u}, Theorem \ref{multiscale_estimate_q} and Lemma \ref{L2Interpolant}, 
we immediately have the following estimate for $u - u_h$:
\begin{corollary}\label{L2estimate_u}
Let the assumptions  of Theorem \ref{multiscale_estimate_u} are fulfilled. Then
\begin{align*}
\|u - u_h\|_0 \le Ch^s (\|u\|_s + \tau^{-1} \|\boldsymbol{q}\|_s) 
         & + C \Big ( \frac{h}{L}\Big )^{\frac{1}{2}} \Big ( \epsilon \|u_0\|_2 + \sqrt{\epsilon} \|\nabla u_0\|_{\infty} + H^{t-1\frac12}\|u_0\|_{t} \Big ) \\
& + C \Big(1 +  \Big (\frac{h}{\tau L}  \Big )^{\frac{1}{2}}\Big) (\|\boldsymbol{e}_q\|_0 + \|e_u -\ehatu\|_{\tau, \mathcal{T}_h}),
\end{align*}
for all $1 \le s \le k+1, 1 \le t \le l+1$.
\end{corollary}

\section{Conclusions}

In this paper, we introduce a hybrid discontinuous Galerkin method for solving
multiscale elliptic equations. This is a first paper in a serious
of two papers. In the present paper,
 we consider polynomial and homogenization-based
coarse-grid spaces and lay a foundation of
hybrid discontinuous Galerkin methods for solving multiscale flow
equations. 
Our method  gives a framework that (1) couples
independently generated multiscale basis functions in each coarse patch
(2) provides a stable global coupling independent of local
discretization, physical scales and contrast
(3) allows avoiding any constraints
on coarse spaces.
Though the coarse spaces in the paper are designed for 
 problems with scale separation, the above  properties of
our framework
are important for extending 
the method to more challenging
multiscale problems with non-separable scales and high contrast.
This is a subject of the subsequent paper.

\section{Acknowledgements}

Y. Efendiev's work is
partially supported by the
DOE and NSF (DMS 0934837 and DMS 0811180).
R. Lazarov's research was supported in parts by NSF (DMS-1016525).

This publication is based in part on work supported by Award 
No. KUS-C1-016-04, made by King Abdullah University of Science 
and Technology (KAUST).

\bibliographystyle{plain} 
\bibliography{bib_raytcho,references,doe_2012}

\end{document}